\documentclass[11pt,oneside,noshowframe]{imsart}

\usepackage[
  a4paper,
  left=30mm,
  right=30mm,
  top=28mm,
  bottom=28mm
]{geometry}

\usepackage{amssymb}
\usepackage{amsfonts,amsmath,amsthm,mathtools}
\usepackage{natbib}
\usepackage{xcolor}
\usepackage{graphicx}
\usepackage{epstopdf}
\usepackage[colorlinks=true,allcolors=blue]{hyperref}

\theoremstyle{plain}
\newtheorem{theorem}{Theorem}
\newtheorem{lemma}{Lemma}
\newtheorem{proposition}{Proposition}

\theoremstyle{definition}

\newtheorem{remark}{Remark}

\startlocaldefs

\newcommand{\EE}{\mathbb{E}}
\newcommand{\VV}{\mathbb{V}}
\newcommand{\bs}[1]{\boldsymbol{#1}}
\newcommand{\Pm}{\mathbf{P}}
\newcommand{\pv}{\mathbf{p}}
\newcommand{\Wm}{\mathbf{W}}
\newcommand{\Sm}{\mathbf{\Sigma}}
\newcommand{\Hm}{\mathbf{H}}
\newcommand{\vecc}{\textup{vec}}
\newcommand{\rvline}{\hspace*{-\arraycolsep}\vline\hspace*{-\arraycolsep}}

\DeclareMathOperator{\tr}{tr}
\DeclareMathOperator{\diag}{diag}
\DeclareMathOperator{\Cov}{Cov}

\endlocaldefs

\begin{document}

\begin{frontmatter}

\title{Matrix asymptotic calculus for plug-in maximum likelihood estimators in finite Markov chains}
\runtitle{Matrix asymptotic calculus for Markov-chain MLEs}
\runauthor{Gavrilopoulos, Trevezas and Votsi}

\begin{aug}

\author{Georgios Gavrilopoulos}
\address{Department of Mathematics, National and Kapodistrian University of Athens\\
Panepistimiopolis, Athens 15784, Greece\\
and Seminar for Statistics, Department of Mathematics, ETH Z\"urich\\
R\"amistrasse 101, Z\"urich 8092, Switzerland}
\ead[label=e0]{}

\author{\fnms{Samis} \snm{Trevezas}\thanksref{t1}}
\address{Department of Mathematics, National and Kapodistrian University of Athens\\
Panepistimiopolis, Athens 15784, Greece\\
and CentraleSup\'elec -- Universit\'e Paris-Saclay\\
3 Rue Joliot Curie, Gif-sur-Yvette 91190, France}
\ead[label=e1]{strevezas@math.uoa.gr}

\author{Irene Votsi}
\address{LIEC, CNRS, Universit\'e de Lorraine\\
8 rue Delestraint, Metz 57000, France}
\ead[label=e2]{}

\end{aug}

\thankstext{t1}{Corresponding author. \printead{e1}}

\begin{abstract}
We develop a matrix asymptotic calculus for plug-in maximum likelihood estimators of smooth characteristics of a finite-state Markov chain observed along a single trajectory. The classical vector central limit theorem for the transition-matrix estimator is represented by a centered Gaussian matrix whose covariance retains the row-stochastic constraints and possible zero transition probabilities. The functional delta method, finite-order Taylor developments and analytic expansions are then realized directly in matrix form. The resulting differential operator yields both the limiting object in its natural matrix, vector or scalar form and, after row-wise vectorization, its covariance matrix. This gives explicit joint distributions for transition probabilities over several steps, finite-dimensional curves of Markov characteristics, stationary and reliability quantities, additive-functional variances, entropy rates and Kemeny-type characteristics. The finite-order developments identify explicit curvature terms, whose Gaussian substitution yields curvature-refined approximations. Matrix differentiation avoids repeated coordinate-wise calculations, often reduces the computational burden, and leads directly to confidence intervals, confidence regions, simultaneous bands and Wald-type procedures.
\end{abstract}

\begin{keyword}[class=KWD]
\kwd{Markov chains}
\kwd{maximum likelihood estimation}
\kwd{Delta method}
\kwd{Gaussian random matrices}
\kwd{asymptotic theory}
\kwd{reliability indicators}
\kwd{non-parametric estimation}
\end{keyword}

\begin{keyword}[class=MSC]
\kwd[Primary ]{62M05}
\kwd[; secondary ]{60J10}
\kwd{62F12}
\kwd{62N05}
\end{keyword}

\end{frontmatter}

\section{Introduction}
\label{sec: introduction}

Non-parametric estimation in finite-state Markov models is a classical topic. The term \emph{non-parametric} means here that the transition probabilities are unknown, while the only structural restriction is that the transition matrix is stochastic. Throughout the paper the state space is $S=\{1,\ldots,s\}$ and the transition matrix is denoted by $\Pm=(p_{ij})_{i,j\in S}$.

Maximum likelihood estimation from one long trajectory observed until time $n$ goes back to \cite{anderson1957statistical} and \cite{billingsley1961}. If $N_i(n)$ is the number of one-step transitions departing from $i$ and $N_{ij}(n)$ is the number of transitions from $i$ to $j$, then the non-parametric MLE is
\begin{equation*}
\widehat p_{ij}(n)=\frac{N_{ij}(n)}{N_i(n)},\qquad i,j\in S,
\end{equation*}
with the usual convention to be null on rows not yet visited. Under irreducibility, this convention is irrelevant asymptotically. The invariance principle of maximum likelihood estimation, in the sense of induced likelihood function as in \cite{zehna1966invariance}, then gives plug-in MLEs for any Markov characteristic which can be expressed as a functional of $\Pm$.

Finite-dimensional parameter estimation for Markov chains and Markov processes has been studied in several directions. For discrete-time Markov chains with continuous state space, one-step MLEs and MLE processes are treated in \cite{VarakinVeretennikov2002,KutoyantsMotrunich2016}. In the context of continuous-time Markov and hidden Markov processes, likelihood methods and related asymptotic problems are discussed in \cite{Chigansky2009,Kutoyants2025}. In finite state spaces, the transition-matrix MLE has a particularly transparent form, but the asymptotic analysis of functionals of $\Pm$ is often carried out after choosing a minimal parametrization. Usually one removes one entry per row and works in a subset of $\mathbb R^{s^2-s}$. This is natural for likelihood differentiation, but it hides the matrix structure of the problem and creates asymmetric formulas for functionals which are intrinsically expressed in terms of $\Pm$.

There is also a boundary issue. Under the true parameter, the chain may be irreducible while some transition probabilities are equal to zero. Then the true parameter lies on the boundary of a reduced Euclidean parametrization, although the matrix functional itself may be regular along the directions selected by the estimator. The first-order fluctuations of the MLE satisfy the row-sum constraints automatically and vanish on entries corresponding to zero transition probabilities. Thus the constraints need not be eliminated before differentiation; they may instead be kept in the limiting Gaussian matrix.

This observation is the starting point of the paper. The usual vector central limit theorem for the transition-matrix MLE is written as a convergence theorem for a centered Gaussian matrix $\Wm_{\!\Pm}$. We use Gaussian random matrices in this elementary vector-valued sense; for background on matrix-variate distributions see, for instance, \cite{gupta2018matrix}. The covariance is retained in its full, generally non-separable, form. The limiting matrix plays for asymptotic distributions the same structural role that $\Pm$ plays for deterministic Markov characteristics. If a characteristic is written as a matrix, vector or scalar functional of $\Pm$, its limiting distribution is obtained by applying the corresponding differential to $\Wm_{\!\Pm}$.

The resulting calculus has two complementary forms. First, it gives the limiting object in the natural form of the characteristic itself: a matrix for transition probabilities over several steps, a vector for a finite-dimensional curve, or a scalar for a reliability index, a long-run variance or an information-theoretic characteristic. Second, if a covariance matrix in vector form is needed, it is obtained afterwards by row-wise vectorization and Kronecker algebra. Thus the matrix representation and the vector representation are two forms of the same expansion. This is the source of the computational simplification: one differentiates once at the matrix level and then vectorizes only the resulting linear operator.

The first-order transfer principle used below is the functional delta method for maps defined on subsets of normed spaces, with tangential set given here by $\mathcal T_{\Pm}$; see \cite[Theorem~20.8]{vaart_1998} and \cite{van1994weak}. The contribution of the present work is therefore not a new abstract delta method. It is the explicit realization of that principle for the non-parametric transition-matrix MLE, together with the finite-order matrix developments and the differential operators required for Markov characteristics. The construction retains the limiting random element as a Gaussian matrix with its full covariance, gives the corresponding row-wise vector covariance by a single operator identity, and provides joint distributions and inferential procedures without separate coordinate-wise derivations.

Kronecker products have long been used in the analysis of Markov chains and multidimensional Markovian systems; see, for instance, \cite{dayar12} and \cite{dayar19}. Kronecker, modular and hierarchical representations of stochastic automata networks and Markov processes were discussed in \cite{buc2004}. \citet{Ang2021} used phase-type distributions and Kronecker algebra to describe the execution of networks of activities modelled by Markov chains. For a survey of continuous-time Markov-chain theory and developments based on matrix analysis, we refer to \cite{Le2022}. In the present paper, Kronecker products are used in a different way: they represent the differential operators which transform the Gaussian limit of the transition-matrix MLE into the Gaussian limits of plug-in estimators. This viewpoint turns several calculations, formerly derived in problem-specific form, into instances of a common asymptotic calculus.

The applications to reliability provide a representative example. Markov and semi-Markov models are widely used for reliability analysis; see \cite{votsi2019,votsi_brouste2019} and the references therein. In the Markov case, quantities such as availability, reliability, failure-time probabilities, mean time to failure, mean time to repair and transition intensities from working to failed states are functionals of the transition matrix. Element-wise asymptotic formulas for some of these quantities were obtained in \cite{sadek_limnios}. The present framework recovers such formulas from a single operator identity and extends them to joint limits of finite-dimensional curves and to further functionals. Variance representations for additive functionals of Markov chains, in the spirit of the Markov-chain central limit theorem and related formulas such as those in \cite{trevezas2009variance}, also fit naturally into the same scheme.

Beyond first-order normal approximations, the matrix calculus gives finite-order stochastic developments. The quadratic term identifies the curvature contribution of the functional at order $n^{-1}$. Replacing the normalized transition-matrix estimator in this term by its Gaussian limit yields the curvature-refined Gaussian approximation examined below. This substitution is not, by itself, a second-order approximation theorem for the distribution; such a conclusion would require an approximation with an explicit rate. A complete order $n^{-1}$ mean correction would also involve the order $n^{-1}$ bias of the transition-matrix MLE.

The paper is organized as follows. Section \ref{sec: prelims} states the classical limit theorem for the transition-matrix MLE, gives its matrix representation, and identifies the linear directions determined by the stochastic constraints. Section \ref{sec: matrix_calculus} develops the matrix-level asymptotic calculus. Its first-order part is related explicitly to the functional delta method, while its finite-order and analytic parts give stochastic developments in the matrix space and the associated covariance operators. Section \ref{sec: polynomial_inverse} treats polynomial and inverse-type functionals, finite-dimensional curves of Markov characteristics, stationary probabilities, additive-functional variances, entropy rates and Kemeny-type characteristics. Section \ref{sec: reliability} derives reliability quantities as explicit instances of the same calculus. Section \ref{sec: numerical} contains numerical illustrations of pointwise inference for availability, simultaneous confidence bands for an availability curve, the limiting normal law of the MTTF estimator, and a curvature-refined Gaussian approximation for the MTTF estimator. Section \ref{sec: discussion} discusses extensions and perspectives. The Kronecker and vectorization identities used in the derivations are recalled in Appendix \ref{app:kronecker}.

\section{Preliminaries}
\label{sec: prelims}

We collect here the probabilistic limit theorem and the geometric notation that will be used throughout the paper. The elementary facts on Kronecker products and vectorization are given in Appendix \ref{app:kronecker}. In the sequel, $\mathbf 1_s=(1,\ldots,1)^\top$ and, for a matrix $\mathbf M$, the vector $\vecc(\mathbf M)$ is obtained by stacking its columns. Since transition matrices are naturally treated row by row, we use the row-wise vectorization
\begin{equation*}
\pv=\vecc(\Pm^\top),
\qquad
\widehat\pv(n)=\vecc\{\widehat\Pm(n)^\top\}.
\end{equation*}
The symbol $\mathbf A$ denotes the $s\times s$ matrix whose entries are all equal to one. We denote by
\begin{equation*}
\mathcal P_{\rm st}=
\left\{\mathbf Q\in\mathbb R^{s\times s}:
\mathbf Q\mathbf 1_s=\mathbf 1_s,\ q_{ij}\geq0\right\}
\end{equation*}
the set of stochastic matrices, and by $\mathcal P_{\rm ir}$ its subset of irreducible stochastic matrices.

\medskip
\noindent\textbf{The classical limit theorem.}
Assume that the true transition matrix $\Pm$ is irreducible, with stationary row vector $\bs\pi=(\pi_1,\ldots,\pi_s)$. The following result is the classical asymptotic normality of the non-parametric MLE of the transition probabilities, written in the row-wise notation used in this paper.

\begin{theorem}
\label{theo:classical_mle_clt}
Let $\Pm\in\mathcal P_{\rm ir}$. Then $\widehat\Pm(n)$ is strongly consistent and
\begin{equation}
\label{eq: MLE_convergence_vec}
\sqrt n\left\{\widehat\pv(n)-\pv\right\}
\xrightarrow[n\rightarrow\infty]{d}
\Wm_{\!\pv}
\sim
\mathcal N_{s^2}(\mathbf 0_{s^2},\Sm_{\pv}),
\end{equation}
where
\begin{equation}
\label{eq:asymptotic_cov}
\Sm_{\pv}=
\begin{pmatrix}
\pi_1^{-1}\mathbf\Lambda_1 & \mathbf 0 & \cdots & \mathbf 0\\
\mathbf 0 & \pi_2^{-1}\mathbf\Lambda_2 & \cdots & \mathbf 0\\
\vdots & \vdots & \ddots & \vdots\\
\mathbf 0 & \mathbf 0 & \cdots & \pi_s^{-1}\mathbf\Lambda_s
\end{pmatrix},
\end{equation}
and
\begin{equation*}
\mathbf\Lambda_i(j,k)=p_{ij}(\delta_{jk}-p_{ik}),
\qquad j,k=1,\ldots,s.
\end{equation*}
Equivalently,
\begin{equation}
\label{eq:mle_d_matrix}
\sqrt n\left\{\widehat\Pm(n)-\Pm\right\}
\xrightarrow[n\rightarrow\infty]{d}
\Wm_{\!\Pm},
\qquad
\vecc(\Wm_{\!\Pm}^{\top})=\Wm_{\!\pv},
\end{equation}
where $\Wm_{\!\Pm}$ is a centered Gaussian matrix characterized by
\begin{equation}
\label{eq:full_cov_tensor}
\operatorname{Cov}(W_{ij},W_{kl})=
\delta_{ik}\frac{1}{\pi_i}p_{ij}(\delta_{jl}-p_{il}),
\qquad i,j,k,l=1,\ldots,s.
\end{equation}
\end{theorem}

The vector formulation of Theorem \ref{theo:classical_mle_clt} is usually obtained by working row by row with a minimal parametrization. For a fixed row $i$, one may remove one transition probability and use the remaining probabilities as free coordinates. Conditionally on the number of visits to $i$, the transition counts from $i$ have the usual multinomial covariance structure. Since $N_i(n)/n$ converges almost surely to $\pi_i$, the limiting covariance of the $i$-th row is $\pi_i^{-1}\mathbf\Lambda_i$. Different rows are asymptotically uncorrelated, which gives the block diagonal matrix in \eqref{eq:asymptotic_cov}. Restoring the removed coordinate gives the full row covariance and also gives the row-sum relation
\begin{equation}
\label{eq:WP_rowsum_zero}
\Wm_{\!\Pm}\mathbf 1_s=\mathbf 0_s
\quad\text{almost surely}.
\end{equation}
Moreover, if $p_{ij}=0$, then $W_{ij}=0$ almost surely. These are not extra assumptions; they are contained in the covariance matrix \eqref{eq:asymptotic_cov}.

The matrix in \eqref{eq:mle_d_matrix} will be used as the limiting Gaussian matrix from which the following asymptotic representations are obtained. Its covariance is kept in full form. In general, a separable representation $\Sm_{\pv}=\mathbf U\otimes\mathbf V$ is not available: it would force all non-zero row covariance blocks $\pi_i^{-1}\mathbf\Lambda_i$ to be proportional to one common matrix. For a general finite Markov chain with $s\geq3$, these blocks need not be proportional.

\medskip
\noindent\textbf{A motivating two-state calculation.}
Let
\begin{equation*}
\Pm=
\begin{pmatrix}
1-p&p\\
q&1-q
\end{pmatrix},
\qquad p,q\in[0,1].
\end{equation*}
Then
\begin{equation*}
\Pm^2=
\begin{pmatrix}
(1-p)^2+pq&p(2-p-q)\\
q(2-p-q)&(1-q)^2+pq
\end{pmatrix}.
\end{equation*}
A minimal parametrization uses $(p,q)$. The corresponding delta-method calculation differentiates the four entries of $\Pm^2$ with respect to these two coordinates. On the other hand, the matrix map $\Phi(\mathbf M)=\mathbf M^2$ has derivative
\begin{equation*}
\Phi'_{\Pm}(\mathbf H)=\Pm\mathbf H+\mathbf H\Pm.
\end{equation*}
The two free limiting coordinates, corresponding to $p_{12}$ and $p_{21}$, are completed into a matrix direction by the row-sum constraints. More precisely, if $X$ and $Y$ denote the limiting coordinates associated with $p_{12}$ and $p_{21}$, respectively, then the corresponding matrix direction is
\begin{equation*}
\mathbf H=
\begin{pmatrix}
-X&X\\
Y&-Y
\end{pmatrix},
\end{equation*}
and therefore $\mathbf H\mathbf{1}_2=\mathbf{0}_2$. Substitution into $\Phi'_{\Pm}$ gives the same linear transformation of $(X,Y)$ as the derivative in the reduced parametrization. Thus the equality of the two first-order laws is not an accident of dimension two. The row-stochastic restrictions are already carried by the limiting matrix direction, and the derivative may be computed in the full matrix space before being evaluated on the tangent directions generated by the estimator.
\medskip

\noindent\textbf{Tangent directions determined by stochasticity.}
The first-order directions compatible with the row-stochastic constraints and the zero entries of $\Pm$ form the tangent space
\begin{equation*}
\mathcal T_{\Pm}=
\{\mathbf H\in\mathbb R^{s\times s}:\mathbf H\mathbf{1}_s=\mathbf{0}_s,
\ h_{ij}=0\ \text{whenever}\ p_{ij}=0\}.
\end{equation*}
By \eqref{eq:WP_rowsum_zero} and the zero-variance entries mentioned above,
\begin{equation}
\label{eq:WP_tangent}
\Wm_{\!\Pm}\in\mathcal T_{\Pm}
\quad\text{almost surely}.
\end{equation}

\medskip
\noindent\textbf{Open classes of matrices.}
The differentiability arguments below are local. It is therefore useful to know that the main classes of transition matrices remain stable under small stochastic variations.

\begin{proposition}
\label{lem:topology_classes}
Let $\mathcal M_s=\mathbb R^{s\times s}$ and let $\mathcal P_{\rm 1p}$ be the subset of stochastic matrices with one positive recurrent class and all remaining states transient. Then $\mathcal P_{\rm 1p}$ and $\mathcal P_{\rm ir}$ are open subsets of $\mathcal P_{\rm st}$ for the relative topology.
\end{proposition}

\begin{proof}
For $\mathcal P_{\rm 1p}$, a standard characterization is the non-singularity of $\mathbf I_s-\Pm+\mathbf A$; see \cite[Proposition 2.14.1]{resnick1992adventures} and \cite{trevezas2009variance}. Since the determinant is continuous, this non-singularity is preserved under sufficiently small stochastic variations within $\mathcal P_{\rm st}$.

For irreducibility, fix $\Pm\in\mathcal P_{\rm ir}$. For every ordered pair $(i,j)$, choose a path from $i$ to $j$ whose transition probabilities are strictly positive. There are finitely many such paths, and hence the minimum of the positive transition probabilities appearing in the selected paths is strictly positive. A sufficiently small stochastic variation within $\mathcal P_{\rm st}$ keeps all these selected transitions positive. The selected paths are therefore preserved, and the varied stochastic matrix remains irreducible.
\end{proof}

All strong consistency assertions in the sequel follow from the strong consistency in Theorem \ref{theo:classical_mle_clt} and the continuous mapping theorem. We shall state strong consistency together with asymptotic normality for the relevant plug-in estimators, without repeating this argument in every proof.


\section{Matrix and analytic stochastic calculus for Markov functionals}
\label{sec: matrix_calculus}

We now pass from the transition-matrix MLE to general Markov characteristics. The limiting Gaussian matrix $\Wm_{\!\Pm}$ plays, for the limiting laws, the same structural role as $\Pm$ plays for the deterministic characteristics. If a characteristic is written as a matrix-valued, vector-valued or scalar-valued functional of $\Pm$, then its limiting object is obtained by applying the corresponding differential to $\Wm_{\!\Pm}$. The row-stochastic restrictions are not ignored in this operation. They are already reflected in the fact that the limiting directions belong to the linear space $\mathcal T_{\Pm}$ introduced in Section \ref{sec: prelims}. 

Let $\mathcal Y$ be a finite-dimensional normed space. A map $\Phi:\mathcal U\subset\mathbb R^{s\times s}\to\mathcal Y$ is Fr\'echet differentiable at $\Pm$ if there exists a linear map $\Phi'_{\Pm}$ such that
\begin{equation*}
\lim_{\lVert\mathbf H\rVert\to0}
\frac{\lVert \Phi(\Pm+\mathbf H)-\Phi(\Pm)-\Phi'_{\Pm}(\mathbf H)\rVert}
{\lVert\mathbf H\rVert}=0.
\end{equation*}
For notational uniformity, we write 
$\Phi_{\Pm}^{(1)}=\Phi'_{\Pm}$.
For $j\geq2$, the $j$-th Fr\'echet differential is denoted by
$\Phi_{\Pm}^{(j)}(\mathbf H_1,\ldots,\mathbf H_j)$. 
When all arguments coincide, we write $\Phi_{\Pm}^{(j)}(\mathbf H,\ldots,\mathbf H)$.

The following theorem collects the transfer principles used throughout the paper. Its first-order assertion is the functional delta method, specialized to the tangential set $\mathcal T_{\Pm}$. The finite-order and analytic assertions follow from Taylor expansion in the finite-dimensional matrix space. The theorem places these principles in a common form from which the explicit Markov-matrix representations, covariance operators and inferential results of the subsequent sections are obtained. The differentiations are performed in $\mathbb R^{s\times s}$, while their values are taken on the support-preserving directions generated by the transition-matrix estimator.

\begin{theorem}
\label{theo:matrix_space_delta}
Let $\Pm\in\mathcal P_{\rm ir}$ and let $\mathcal U\subset\mathbb R^{s\times s}$ be an open neighbourhood of $\Pm$. Let
\begin{equation*}
\phi:\mathcal U\cap\mathcal P_{\rm st}\longrightarrow\mathcal Y
\end{equation*}
be a functional defined on stochastic matrices, where $\mathcal Y$ is a finite-dimensional normed space. Suppose that there exists a map
\begin{equation*}
\Phi:\mathcal U\longrightarrow\mathcal Y
\end{equation*}
such that
\begin{equation*}
\Phi(\mathbf Q)=\phi(\mathbf Q)
\end{equation*}
for every $\mathbf Q\in\mathcal U\cap\mathcal P_{\rm st}$. Set
\begin{equation}
\label{eq:def_Wn}
\Wm_n=\sqrt n\{\widehat\Pm(n)-\Pm\}.
\end{equation}
If $\Phi$ is continuous at $\Pm$, then
\begin{equation}
\label{eq:central_consistency}
\phi(\widehat\Pm(n))\longrightarrow\phi(\Pm)
\quad\text{a.s.}
\end{equation}
where the plug-in estimator may be defined arbitrarily before all rows have been observed. In particular, \eqref{eq:central_consistency} holds under each of the differentiability assumptions below. The following assertions also hold.
\begin{itemize}
\item[(i)] If $\Phi$ is Fr\'{e}chet differentiable at $\Pm$, then
\begin{equation}
\label{eq:central_first_order}
\sqrt n\{\phi(\widehat\Pm(n))-\phi(\Pm)\}
\xrightarrow[n\rightarrow\infty]{d}
\Phi'_{\Pm}(\Wm_{\!\Pm}).
\end{equation}
\item[(ii)] If $\Phi$ is $r$ times continuously Fr\'{e}chet differentiable on $\mathcal U$, with $r\geq2$, then
\begin{equation}
\label{eq:central_finite_expansion}
\phi(\widehat\Pm(n))
=
\phi(\Pm)
+
\sum_{j=1}^{r}
\frac{1}{j!\,n^{j/2}}
\Phi_{\Pm}^{(j)}(\Wm_n,\ldots,\Wm_n)
+
o_p(n^{-r/2}).
\end{equation}
Moreover, for every $m=1,\ldots,r$,
\begin{align}
& n^{m/2}
\left\{
\phi(\widehat\Pm(n))-\phi(\Pm)
-
\sum_{j=1}^{m-1}
\frac{1}{j!\,n^{j/2}}
\Phi_{\Pm}^{(j)}(\Wm_n,\ldots,\Wm_n)
\right\}
\nonumber\\
&\hspace{4cm}
\xrightarrow[n\rightarrow\infty]{d}
\frac1{m!}\Phi_{\Pm}^{(m)}(\Wm_{\!\Pm},\ldots,\Wm_{\!\Pm}).
\label{eq:central_m_order_limit}
\end{align}
where the sum is empty when $m=1$.
\item[(iii)] If $\Phi$ is real analytic in a neighbourhood of $\Pm$ contained in $\mathcal U$, then there exists an open neighbourhood $\mathcal V\subset\mathcal U$ of $\Pm$ such that, on the event $\{\widehat\Pm(n)\in\mathcal V\cap\mathcal P_{\rm st}\}$,
\begin{equation}
\label{eq:central_analytic_series}
\phi(\widehat\Pm(n))
=
\phi(\Pm)
+
\sum_{j\geq1}
\frac{1}{j!\,n^{j/2}}
\Phi_{\Pm}^{(j)}(\Wm_n,\ldots,\Wm_n),
\end{equation}
where the series converges absolutely. The probability of this event tends to one. Consequently, for every fixed $r\geq1$,
\begin{equation}
\label{eq:central_analytic_truncation}
\phi(\widehat\Pm(n))
=
\phi(\Pm)
+
\sum_{j=1}^{r}
\frac{1}{j!\,n^{j/2}}
\Phi_{\Pm}^{(j)}(\Wm_n,\ldots,\Wm_n)
+
O_p(n^{-(r+1)/2}).
\end{equation}
\item[(iv)] If, for some integer $r\geq1$, two $C^r$ maps $\Phi_1$ and $\Phi_2$ represent the same stochastic functional $\phi$ on $\mathcal U\cap\mathcal P_{\rm st}$, then, for every $j=1,\ldots,r$,
\begin{equation}
\label{eq:central_invariance}
\Phi_{1,\Pm}^{(j)}(\mathbf H_1,\ldots,\mathbf H_j)
=
\Phi_{2,\Pm}^{(j)}(\mathbf H_1,\ldots,\mathbf H_j),
\qquad
\mathbf H_1,\ldots,\mathbf H_j\in\mathcal T_{\Pm}.
\end{equation}
Consequently, the terms in \eqref{eq:central_first_order}--\eqref{eq:central_analytic_truncation}, whenever evaluated on $\Wm_n$ or $\Wm_{\!\Pm}$, do not depend on the chosen differentiable matrix representation of $\phi$.
\end{itemize}
\end{theorem}

\begin{proof}
Let
\begin{equation*}
E_n=\{\widehat\Pm(n)\in\mathcal U\}
\quad\text{and}\quad
F_n=\{N_i(n)>0,\ i=1,\ldots,s\}.
\end{equation*}
The strong consistency in Theorem~\ref{theo:classical_mle_clt} gives $\mathbb P(E_n)\to1$. Since the chain is finite and irreducible, every state is visited infinitely often almost surely; hence $F_n$ occurs eventually almost surely. On $E_n\cap F_n$, the matrix $\widehat\Pm(n)$ is stochastic and
\begin{equation*}
\phi(\widehat\Pm(n))=\Phi(\widehat\Pm(n)).
\end{equation*}
If $\Phi$ is continuous at $\Pm$, \eqref{eq:central_consistency} follows from the continuous mapping theorem.

We first establish the tangential differentiability required in (i). Let $t_n\downarrow0$ and let $\mathbf H_n\to\mathbf H\in\mathcal T_{\Pm}$ be such that
\begin{equation*}
\Pm+t_n\mathbf H_n\in\mathcal U\cap\mathcal P_{\rm st}.
\end{equation*}
Since $\Phi$ represents $\phi$ on this set and is Fr\'{e}chet differentiable at $\Pm$,
\begin{align*}
\frac{\phi(\Pm+t_n\mathbf H_n)-\phi(\Pm)}{t_n}
&=
\frac{\Phi(\Pm+t_n\mathbf H_n)-\Phi(\Pm)}{t_n}\\
&=
\Phi'_{\Pm}(\mathbf H_n)+o(1)
\longrightarrow
\Phi'_{\Pm}(\mathbf H).
\end{align*}
Thus $\phi$ is Hadamard differentiable at $\Pm$ tangentially to $\mathcal T_{\Pm}$, with derivative equal to the restriction of $\Phi'_{\Pm}$ to $\mathcal T_{\Pm}$. By Theorem~\ref{theo:classical_mle_clt},
\begin{equation*}
\Wm_n\xrightarrow[n\rightarrow\infty]{d}\Wm_{\!\Pm}
\quad\text{and}\quad
\Wm_{\!\Pm}\in\mathcal T_{\Pm}\quad\text{a.s.}
\end{equation*}
Assertion (i) now follows from the functional delta method \cite[Theorem~20.8]{vaart_1998}.

For (ii), put
\begin{equation*}
\mathbf H_n=\widehat\Pm(n)-\Pm=n^{-1/2}\Wm_n.
\end{equation*}
Taylor's formula in the finite-dimensional normed space $\mathbb R^{s\times s}$ gives
\begin{equation*}
\Phi(\Pm+\mathbf H_n)
=
\Phi(\Pm)
+
\sum_{j=1}^{r}
\frac1{j!}\Phi_{\Pm}^{(j)}(\mathbf H_n,\ldots,\mathbf H_n)
+
\mathbf R_r(\mathbf H_n),
\end{equation*}
where
\begin{equation*}
\lVert\mathbf R_r(\mathbf H_n)\rVert
=o_p(\lVert\mathbf H_n\rVert^r).
\end{equation*}
Since $\Wm_n=O_p(1)$, it follows that
\begin{equation*}
\mathbf R_r(\mathbf H_n)=o_p(n^{-r/2}),
\end{equation*}
which proves \eqref{eq:central_finite_expansion}.

Let $m\in\{1,\ldots,r\}$. Multiplying the expansion by $n^{m/2}$ after subtracting the terms of orders $1,\ldots,m-1$ yields
\begin{align*}
& n^{m/2}
\left\{
\phi(\widehat\Pm(n))-\phi(\Pm)
-
\sum_{j=1}^{m-1}
\frac{1}{j!\,n^{j/2}}
\Phi_{\Pm}^{(j)}(\Wm_n,\ldots,\Wm_n)
\right\}\\
&\quad=
\frac1{m!}\Phi_{\Pm}^{(m)}(\Wm_n,\ldots,\Wm_n)
+
\sum_{j=m+1}^{r}
\frac{1}{j!\,n^{(j-m)/2}}
\Phi_{\Pm}^{(j)}(\Wm_n,\ldots,\Wm_n)
+
o_p(1),
\end{align*}
where the second sum is absent if $m=r$. Since $\Wm_n=O_p(1)$, every term in this sum is $o_p(1)$. The mapping
\begin{equation*}
\mathbf H\longmapsto
\Phi_{\Pm}^{(m)}(\mathbf H,\ldots,\mathbf H)
\end{equation*}
is continuous. The continuous mapping theorem therefore proves \eqref{eq:central_m_order_limit}.

For (iii), real analyticity at $\Pm$ gives an open neighbourhood $\mathcal V\subset\mathcal U$ on which the Taylor series of $\Phi$ at $\Pm$ converges absolutely to $\Phi$. Strong consistency and the eventual occurrence of $F_n$ imply
\begin{equation*}
\mathbb P\{\widehat\Pm(n)\in\mathcal V\cap\mathcal P_{\rm st}\}
\longrightarrow1.
\end{equation*}
On this event,
\begin{equation*}
\Phi(\widehat\Pm(n))
=
\Phi(\Pm)
+
\sum_{j\geq1}
\frac1{j!}
\Phi_{\Pm}^{(j)}(\mathbf H_n,\ldots,\mathbf H_n).
\end{equation*}
Substitution of $\mathbf H_n=n^{-1/2}\Wm_n$ proves \eqref{eq:central_analytic_series}. By reducing $\mathcal V$ if necessary, the analytic remainder satisfies, for every fixed $r\geq1$,
\begin{equation*}
\left\|
\sum_{j\geq r+1}
\frac1{j!}
\Phi_{\Pm}^{(j)}(\mathbf H_n,\ldots,\mathbf H_n)
\right\|
\leq C_r\lVert\mathbf H_n\rVert^{r+1}
\end{equation*}
on an event whose probability tends to one. Since $\Wm_n=O_p(1)$, this gives \eqref{eq:central_analytic_truncation}.

It remains to prove (iv). Let $\mathbf H_1,\ldots,\mathbf H_j\in\mathcal T_{\Pm}$. By the definition of $\mathcal T_{\Pm}$ in Section~\ref{sec: prelims}, every row of each $\mathbf H_\ell$ sums to zero and $H_{\ell,ab}=0$ whenever $p_{ab}=0$. Hence, for all sufficiently small real $t_1,\ldots,t_j$,
\begin{equation*}
\Pm+t_1\mathbf H_1+\cdots+t_j\mathbf H_j
\in\mathcal U\cap\mathcal P_{\rm st}.
\end{equation*}
The equality of $\Phi_1$ and $\Phi_2$ on this set implies
\begin{equation*}
\Phi_1(\Pm+t_1\mathbf H_1+\cdots+t_j\mathbf H_j)
=
\Phi_2(\Pm+t_1\mathbf H_1+\cdots+t_j\mathbf H_j).
\end{equation*}
Differentiating once with respect to each of $t_1,\ldots,t_j$ at the origin proves \eqref{eq:central_invariance}. Since $\Wm_n\in\mathcal T_{\Pm}$ eventually almost surely and $\Wm_{\!\Pm}\in\mathcal T_{\Pm}$ almost surely, the final assertion follows.
\end{proof}

The second-order consequences of Theorem \ref{theo:matrix_space_delta} are recorded next.
\begin{remark}
Taking $m=2$ in assertion~(ii) gives
\begin{equation}
\label{eq:second_order_expansion}
\phi(\widehat\Pm(n))
=
\phi(\Pm)
+
\frac1{\sqrt n}\Phi'_{\Pm}(\Wm_n)
+
\frac1{2n}\Phi_{\Pm}^{(2)}(\Wm_n,\Wm_n)
+
o_p(n^{-1}),
\end{equation}
and
\begin{equation}
\label{eq:second_order_limit}
n\left\{
\phi(\widehat\Pm(n))-\phi(\Pm)-\Phi'_{\Pm}(\widehat\Pm(n)-\Pm)
\right\}
\xrightarrow[n\rightarrow\infty]{d}
\frac12\Phi_{\Pm}^{(2)}(\Wm_{\!\Pm},\Wm_{\!\Pm}).
\end{equation}
The term $\frac12\Phi_{\Pm}^{(2)}(\mathbf H,\mathbf H)$ is the curvature contribution of the functional. The stochastic expansion \eqref{eq:second_order_expansion} is expressed in terms of $\Wm_n$. Replacing $\Wm_n$ by $\Wm_{\!\Pm}$ gives a curvature-refined Gaussian approximation. The central limit theorem alone does not establish second-order distributional accuracy for this approximation; such a conclusion would require an Edgeworth expansion or another approximation with an explicit rate.

For a scalar functional, put
\begin{equation*}
\mathbf H_n=\widehat\Pm(n)-\Pm
\end{equation*}
and define the second-order remainder by
\begin{equation*}
R_2(n)=
\phi(\widehat\Pm(n))-\phi(\Pm)
-\Phi'_{\Pm}(\mathbf H_n)
-\frac12\Phi_{\Pm}^{(2)}(\mathbf H_n,\mathbf H_n).
\end{equation*}
If, in addition,
\begin{equation*}
n\,\EE(\mathbf H_n)\longrightarrow\mathbf B_{\Pm},
\end{equation*}
the family $\{\lVert\Wm_n\rVert^2:n\geq1\}$ is uniformly integrable, and
\begin{equation*}
n\,\EE\{|R_2(n)|\}\longrightarrow0,
\end{equation*}
then expectations may be taken in \eqref{eq:second_order_expansion}, and
\begin{equation*}
n\left[\EE\{\phi(\widehat\Pm(n))\}-\phi(\Pm)\right]
\longrightarrow
\Phi'_{\Pm}(\mathbf B_{\Pm})
+
\frac12\EE\left\{\Phi_{\Pm}^{(2)}(\Wm_{\!\Pm},\Wm_{\!\Pm})\right\}.
\end{equation*}
The first term is inherited from the order $n^{-1}$ bias of the transition-matrix estimator, whereas the second is the curvature contribution of the functional. Thus the curvature term alone is not a complete order $n^{-1}$ mean correction unless the matrix-bias contribution vanishes or is also included.

For the main families used later, the curvature term has an explicit matrix form. For $\Phi_k(\Pm)=\Pm^k$,
\begin{equation*}
\frac12\Phi_{k,\Pm}^{(2)}(\mathbf H,\mathbf H)=
\sum_{a+b+c=k-2}\Pm^a\mathbf H\Pm^b\mathbf H\Pm^c,
\end{equation*}
where the sum is over non-negative integers $a,b,c$. For the resolvent
\begin{equation*}
\mathbf N_{\mathbf B,\mathbf C}(\Pm)=(\mathbf I_r-\mathbf B\Pm\mathbf C)^{-1},
\end{equation*}
the corresponding term is
\begin{equation*}
\mathbf N_{\mathbf B,\mathbf C}\mathbf B\mathbf H\mathbf C
\mathbf N_{\mathbf B,\mathbf C}\mathbf B\mathbf H\mathbf C
\mathbf N_{\mathbf B,\mathbf C}.
\end{equation*}
For the stationary vector, with
\begin{equation*}
\mathbf R_{\Pm}=(\mathbf I_s-\Pm+\mathbf A)^{-1},
\end{equation*}
the curvature term is
\begin{equation*}
\boldsymbol\pi\mathbf H\mathbf R_{\Pm}\mathbf H\mathbf R_{\Pm}.
\end{equation*}
Scalar reliability and Markov characteristics are obtained by multiplying these expressions by the corresponding fixed vectors and selection matrices. Thus the curvature term of an availability ordinate $A_k=\mathbf a\Pm^k\mathbf e_U$ is
\begin{equation*}
\mathbf a
\left\{
\sum_{a+b+c=k-2}\Pm^a\mathbf H\Pm^b\mathbf H\Pm^c
\right\}
\mathbf e_U,
\end{equation*}
whereas for the mean time to failure, with $\mathbf N_U=(\mathbf I_r-\Pm_{UU})^{-1}$ and $\mathbf H_{UU}=\mathbf E_U^\top\mathbf H\mathbf E_U$, it is
\begin{equation*}
\mathbf a_U\mathbf N_U\mathbf H_{UU}\mathbf N_U\mathbf H_{UU}\mathbf N_U\mathbf 1_r.
\end{equation*}
The same expression with the down-state block gives the curvature term for the mean time to repair. For the intensity functional $r_k=\mathbf a\Pm^{k-1}\mathbf J_U\Pm\mathbf e_D$, the curvature term is
\begin{equation*}
\mathbf a\left\{
\sum_{a+b+c=k-3}\Pm^a\mathbf H\Pm^b\mathbf H\Pm^c\mathbf J_U\Pm
+\sum_{i=0}^{k-2}\Pm^i\mathbf H\Pm^{k-2-i}\mathbf J_U\mathbf H
\right\}\mathbf e_D,
\end{equation*}
with the convention that empty sums are zero.
\end{remark}

\begin{remark}
When $\Phi$ is analytic, the expansion in \eqref{eq:central_analytic_series} gives a systematic hierarchy of approximations in powers of $n^{-1/2}$. Polynomial functionals have a finite expansion. Inverse-type functionals have a convergent local expansion on the set where the defining inverse exists. This is the reason why powers, resolvents, stationary probabilities, mean hitting and repair times, and the reliability characteristics below are covered by the same theorem.
\end{remark}

The next proposition gives the form in which the matrix limit is converted into the vector covariance matrix used for statistical inference. It also records the common situation where the derivative is a finite sum of left and right matrix multiplications.

\begin{proposition}
\label{prop:linear_covariance}
Let $\Phi$ be as in Theorem \ref{theo:matrix_space_delta}, with values in $\mathbb R^{q\times r}$, and suppose that there exists a matrix $\mathbf L_\Phi\in\mathbb R^{qr\times s^2}$ such that
\begin{equation}
\label{eq:operator_linearization}
\vecc\{\Phi'_{\Pm}(\mathbf H)^\top\}=
\mathbf L_\Phi\vecc(\mathbf H^\top),
\qquad \mathbf H\in\mathbb R^{s\times s}.
\end{equation}
Then
\begin{equation*}
\vecc\{\Phi'_{\Pm}(\Wm_{\!\Pm})^\top\}
=
\mathbf L_\Phi\Wm_{\!\pv},
\end{equation*}
and
\begin{equation}
\label{eq:linear_covariance}
\operatorname{Cov}\left[\vecc\{\Phi'_{\Pm}(\Wm_{\!\Pm})^\top\}\right]
=
\mathbf L_\Phi\Sm_{\pv}\mathbf L_\Phi^\top.
\end{equation}
If
\begin{equation}
\label{eq:operator_sum_form}
\Phi'_{\Pm}(\mathbf H)=
\sum_{\ell=1}^{m}\mathbf B_\ell\mathbf H\mathbf A_\ell,
\end{equation}
with fixed matrices $\mathbf B_\ell\in\mathbb R^{q\times s}$ and $\mathbf A_\ell\in\mathbb R^{s\times r}$, then
\begin{equation}
\label{eq:explicit_L_phi}
\mathbf L_\Phi=\sum_{\ell=1}^{m}\mathbf B_\ell\otimes\mathbf A_\ell^\top.
\end{equation}
Finally, for a scalar characteristic
\begin{equation*}
\phi(\Pm)=\mathbf a\Phi(\Pm)\mathbf c,
\end{equation*}
where $\mathbf a\in\mathbb R^{1\times q}$ and $\mathbf c\in\mathbb R^{r\times1}$ are fixed, the limiting variable is
\begin{equation*}
\mathbf a\Phi'_{\Pm}(\Wm_{\!\Pm})\mathbf c,
\end{equation*}
with variance
\begin{equation}
\label{eq:scalar_linear_variance}
\sigma^2_\phi=
(\mathbf a\otimes\mathbf c^\top)
\mathbf L_\Phi\Sm_{\pv}\mathbf L_\Phi^\top
(\mathbf a^\top\otimes\mathbf c).
\end{equation}
\end{proposition}

\begin{proof}
The first covariance identity follows from the definition of $\mathbf L_\Phi$ and from $\operatorname{Cov}(\Wm_{\!\pv})=\Sm_{\pv}$. If the derivative has the form \eqref{eq:operator_sum_form}, Lemma \ref{lem: vec_propereties} gives
\begin{equation*}
\vecc\{(\mathbf B_\ell\mathbf H\mathbf A_\ell)^\top\}=(\mathbf B_\ell\otimes\mathbf A_\ell^\top)\vecc(\mathbf H^\top),
\end{equation*}
and summation gives \eqref{eq:explicit_L_phi}. The scalar formula follows from
\begin{equation*}
\mathbf a\mathbf Y\mathbf c=(\mathbf a\otimes\mathbf c^\top)\vecc(\mathbf Y^\top),
\end{equation*}
applied to $\mathbf Y=\Phi'_{\Pm}(\Wm_{\!\Pm})$.
\end{proof}

\begin{remark}
\label{rem:plug_in_inference}
The same matrix $\mathbf L_\Phi$ gives the standard inferential quantities. Let
\begin{equation*}
\bs{\phi}=\vecc\{\Phi(\Pm)^\top\},
\qquad
\widehat{\bs{\phi}}_n=\vecc\{\Phi(\widehat\Pm(n))^\top\},
\qquad
\mathbf\Gamma_\Phi=\mathbf L_\Phi\Sm_{\pv}\mathbf L_\Phi^\top.
\end{equation*}
Unknown quantities in $\mathbf\Gamma_\Phi$ are estimated by replacing $\Pm$ and $\boldsymbol\pi$ by their MLEs. If $\mathbf\Gamma_\Phi$ is non-singular and $\widehat{\mathbf\Gamma}_{\Phi,n}$ is consistent, then
\begin{equation*}
n(\widehat{\bs{\phi}}_n-\bs{\phi})^\top
\widehat{\mathbf\Gamma}_{\Phi,n}^{-1}
(\widehat{\bs{\phi}}_n-\bs{\phi})
\xrightarrow[n\rightarrow\infty]{d}\chi^2_{qr}.
\end{equation*}
For a scalar characteristic with asymptotic variance $\sigma_\phi^2>0$, the usual interval is
\begin{equation*}
\widehat\phi_n\pm z_{1-\alpha/2}\frac{\widehat\sigma_{\phi,n}}{\sqrt n}.
\end{equation*}
Suppose that $\bs{\phi}=(\phi_0,\ldots,\phi_m)^\top$ is a fixed finite-dimensional curve and that the diagonal elements of $\mathbf\Gamma_\Phi$ are positive. Define
\begin{equation*}
\widehat{\mathbf D}_{\Phi,n}
=
\diag(\widehat\sigma_{\phi_0,n},\ldots,\widehat\sigma_{\phi_m,n}),
\qquad
\widehat{\mathbf C}_{\Phi,n}
=
\widehat{\mathbf D}_{\Phi,n}^{-1}
\widehat{\mathbf\Gamma}_{\Phi,n}
\widehat{\mathbf D}_{\Phi,n}^{-1}.
\end{equation*}
Let $\widehat c_{1-\alpha,n}$ be the conditional $(1-\alpha)$-quantile of
\begin{equation*}
\max_{0\leq k\leq m}|Z_k|,
\qquad
\mathbf Z\sim\mathcal N_{m+1}
(\mathbf 0_{m+1},\widehat{\mathbf C}_{\Phi,n}).
\end{equation*}
If $\widehat{\mathbf\Gamma}_{\Phi,n}\xrightarrow{p}\mathbf\Gamma_\Phi$ and the distribution of the corresponding Gaussian maximum is continuous at its $(1-\alpha)$-quantile, then
\begin{equation*}
\widehat\phi_{k,n}\pm \widehat c_{1-\alpha,n}
\frac{\widehat\sigma_{\phi_k,n}}{\sqrt n},
\qquad k=0,\ldots,m,
\end{equation*}
have asymptotic joint coverage $1-\alpha$. Thus both the marginal scales and the simultaneous critical value are computed from the same plug-in covariance matrix. If $\mathbf R\bs{\phi}=\mathbf b$ is a system of $d$ linear restrictions and $\mathbf R\mathbf\Gamma_\Phi\mathbf R^\top$ is non-singular, then, under the null hypothesis,
\begin{equation*}
n(\mathbf R\widehat{\bs{\phi}}_n-\mathbf b)^\top
(\mathbf R\widehat{\mathbf\Gamma}_{\Phi,n}\mathbf R^\top)^{-1}
(\mathbf R\widehat{\bs{\phi}}_n-\mathbf b)
\xrightarrow[n\rightarrow\infty]{d}\chi^2_d.
\end{equation*}
\end{remark}

We shall also use the following elementary differential rules.

\begin{proposition}
\label{prop:product_inverse_rules}
Let $\Phi$ and $\Psi$ be Fr\'echet differentiable matrix-valued maps at $\Pm$. Then
\begin{equation*}
(\Phi\Psi)'_{\Pm}(\mathbf H)
=\Phi'_{\Pm}(\mathbf H)\Psi(\Pm)+\Phi(\Pm)\Psi'_{\Pm}(\mathbf H).
\end{equation*}
If $\Phi(\Pm)$ is non-singular, then the inverse map is differentiable at $\Pm$ and
\begin{equation*}
(\Phi^{-1})'_{\Pm}(\mathbf H)
=-\Phi(\Pm)^{-1}\Phi'_{\Pm}(\mathbf H)\Phi(\Pm)^{-1}.
\end{equation*}
\end{proposition}

\begin{proof}
For the product rule, write
\begin{equation*}
\Phi(\Pm+\mathbf H)\Psi(\Pm+\mathbf H)-\Phi(\Pm)\Psi(\Pm)
\end{equation*}
as
\begin{align*}
&\{\Phi(\Pm+\mathbf H)-\Phi(\Pm)\}\Psi(\Pm)
+\Phi(\Pm)\{\Psi(\Pm+\mathbf H)-\Psi(\Pm)\}\\
&\qquad
+\{\Phi(\Pm+\mathbf H)-\Phi(\Pm)\}
 \{\Psi(\Pm+\mathbf H)-\Psi(\Pm)\}.
\end{align*}
The first two terms give the derivative and the last term is of second order. The inverse rule follows by differentiating the identity $\Phi(\Pm)\Phi(\Pm)^{-1}=\mathbf I$ in the direction $\mathbf H$.
\end{proof}

\section{Polynomial and inverse-type functionals}
\label{sec: polynomial_inverse}

The general calculus becomes useful when the differential can be written explicitly. In finite Markov theory this happens for the main algebraic families: powers, products with fixed selectors, resolvents and stationary distributions. In each case we first keep the limiting random object in its natural form and then give the vector covariance obtained by Kronecker algebra. This organization avoids deriving a new element-wise covariance formula for every functional.

\subsection{Powers of the transition matrix}
\label{subsec:powers}

The first family consists of transition probabilities over a fixed number of steps. For each fixed integer $k\geq1$, define $\Phi_k(\Pm)=\Pm^k$. The following proposition gives simultaneously the random matrix limit and the row-wise vector covariance used for inference.

\begin{proposition}
\label{prop:powers}
The map $\Phi_k:\mathbb R^{s\times s}\to\mathbb R^{s\times s}$, $\Phi_k(\mathbf M)=\mathbf M^k$, is Fr\'echet differentiable and
\begin{equation}
\label{eq:derivative_power}
\Phi'_{k,\Pm}(\mathbf H)=
\sum_{i=0}^{k-1}\Pm^i\mathbf H\Pm^{k-1-i}.
\end{equation}
Consequently,
\begin{equation}
\label{eq:limit_power_matrix}
\sqrt n\left\{\widehat\Pm(n)^k-\Pm^k\right\}
\xrightarrow[n\rightarrow\infty]{d}
\Wm_{\!\Pm^k}:=
\sum_{i=0}^{k-1}\Pm^i\Wm_{\!\Pm}\Pm^{k-1-i}.
\end{equation}
If
\begin{equation}
\label{Mk}
\mathbf M_k=
\sum_{i=0}^{k-1}\Pm^i\otimes(\Pm^{k-1-i})^\top,
\end{equation}
then
\begin{equation}
\label{vecPk}
\Wm_{\!\pv^{(k)}}:=\vecc(\Wm_{\!\Pm^k}^{\top})
=\mathbf M_k\Wm_{\!\pv}
\sim
\mathcal N_{s^2}(\mathbf 0_{s^2},\Sm_{\pv^{(k)}}),
\end{equation}
where
\begin{equation}
\label{eq:explicit_var_P^k}
\Sm_{\pv^{(k)}}=
\mathbf M_k\Sm_{\pv}\mathbf M_k^\top.
\end{equation}
\end{proposition}

\begin{proof}
For $k=1$ the formula is immediate. For $k\geq2$, expand the non-commutative product $(\Pm+\mathbf H)^k$. The first-order terms are obtained by placing one copy of $\mathbf H$ at each possible position in the product and keeping $\Pm$ in all remaining positions. This gives exactly
\begin{equation*}
\sum_{i=0}^{k-1}\Pm^i\mathbf H\Pm^{k-1-i}.
\end{equation*}
All other terms contain at least two factors $\mathbf H$ and are $O(\lVert\mathbf H\rVert^2)$ in any matrix norm. This proves the Fr\'echet derivative. The matrix limit \eqref{eq:limit_power_matrix} is then Theorem \ref{theo:matrix_space_delta} applied to $\Phi_k$.

It remains only to write the same limit in row-wise vector form. Taking transposes in \eqref{eq:derivative_power} gives terms of the form $(\Pm^{k-1-i})^\top\mathbf H^\top(\Pm^i)^\top$. Lemma \ref{lem: vec_propereties} therefore gives the multiplier $\Pm^i\otimes(\Pm^{k-1-i})^\top$ for the $i$-th term, and summing over $i$ gives $\mathbf M_k$. Since $\vecc(\Wm_{\!\Pm}^\top)=\Wm_{\!\pv}$, the covariance is $\mathbf M_k\Sm_{\pv}\mathbf M_k^\top$.
\end{proof}

The convention $\mathbf M_0=\mathbf 0_{s^2\times s^2}$ and $\Wm_{\!\Pm^0}=\mathbf 0_{s\times s}$ will be used when $\Pm^0=\mathbf I_s$ is deterministic.

\subsection{Finite-dimensional curves of Markov characteristics}
\label{subsec:curves}

A number of statistical objects are not single values but finite pieces of curves. Let $\mathbf a$ be a fixed initial row vector and let $\mathbf c\in\mathbb R^{s\times1}$. For $m\geq0$, define
\begin{equation}
\label{eq:general_curve}
\Gamma_m(\Pm;\mathbf a,\mathbf c)=
\big(\mathbf a\Pm^0\mathbf c,
\mathbf a\Pm^1\mathbf c,
\ldots,
\mathbf a\Pm^m\mathbf c\big)^\top.
\end{equation}
The vector $\Gamma_m$ contains, for example, an availability curve, a cumulative occupation curve, or any finite family of transition probabilities selected by $\mathbf a$ and $\mathbf c$.

\begin{proposition}
\label{prop:markov_curves}
The plug-in estimator $\Gamma_m(\widehat\Pm(n);\mathbf a,\mathbf c)$ is strongly consistent and
\begin{equation}
\label{eq:curve_limit}
\sqrt n\left\{\Gamma_m(\widehat\Pm(n);\mathbf a,\mathbf c)-\Gamma_m(\Pm;\mathbf a,\mathbf c)\right\}
\xrightarrow[n\rightarrow\infty]{d}
\mathbf Z_m,
\end{equation}
where
\begin{equation*}
Z_{m,k}=\mathbf a\Wm_{\!\Pm^k}\mathbf c,
\qquad k=0,\ldots,m.
\end{equation*}
The covariance matrix of $\mathbf Z_m$ is given by
\begin{equation}
\label{eq:curve_cov}
\Cov(Z_{m,k},Z_{m,\ell})=
(\mathbf a\otimes\mathbf c^\top)
\mathbf M_k\Sm_{\pv}\mathbf M_\ell^\top
(\mathbf a^\top\otimes\mathbf c),
\qquad k,\ell=0,\ldots,m.
\end{equation}
\end{proposition}

\begin{proof}
For $k\geq1$, the $k$-th coordinate of \eqref{eq:general_curve} is the scalar linear form of the matrix power $\Pm^k$, namely $\mathbf a\Pm^k\mathbf c$. Proposition \ref{prop:powers} gives the first-order term of $\Pm^k$, and fixed multiplication by $\mathbf a$ and $\mathbf c$ gives the limiting variable $\mathbf a\Wm_{\!\Pm^k}\mathbf c$. The coordinate $k=0$ is deterministic, since $\Pm^0=\mathbf I_s$, and therefore its first-order term is zero. For any deterministic vector $\boldsymbol\lambda=(\lambda_0,\ldots,\lambda_m)$, the linear combination of the coordinates is again a finite sum of scalar linear forms of powers of $\Pm$; the preceding argument gives its limit. The Cram\'er--Wold device gives joint convergence. Finally, the covariance between coordinates $k$ and $\ell$ is obtained by writing both limiting terms in row-wise vector form,
\begin{equation*}
\mathbf a\Wm_{\!\Pm^k}\mathbf c
=(\mathbf a\otimes\mathbf c^\top)\mathbf M_k\Wm_{\!\pv},
\end{equation*}
and similarly for $\ell$. This gives \eqref{eq:curve_cov}.
\end{proof}

\subsection{Resolvents and inverse-type functionals}
\label{subsec:inverse_type}

Inverse-type functionals enter naturally through fundamental matrices of transient subchains, restricted transition matrices and stationary probabilities. The following proposition gives the inverse-type form used later for mean hitting and repair times.

\begin{proposition}
\label{prop:resolvent}
Let $\mathbf B\in\mathbb R^{r\times s}$ and $\mathbf C\in\mathbb R^{s\times r}$ be fixed matrices. Suppose that $\rho(\mathbf B\Pm\mathbf C)<1$ and define
\begin{equation}
\label{eq:general_resolvent}
\mathbf N_{\mathbf B,\mathbf C}(\Pm)=
(\mathbf I_r-\mathbf B\Pm\mathbf C)^{-1}.
\end{equation}
Then
\begin{equation}
\label{eq:resolvent_derivative}
\mathbf N'_{\mathbf B,\mathbf C;\Pm}(\mathbf H)=
\mathbf N_{\mathbf B,\mathbf C}(\Pm)
\mathbf B\mathbf H\mathbf C
\mathbf N_{\mathbf B,\mathbf C}(\Pm),
\end{equation}
and
\begin{equation}
\label{eq:resolvent_limit}
\sqrt n\left\{\mathbf N_{\mathbf B,\mathbf C}(\widehat\Pm(n))-\mathbf N_{\mathbf B,\mathbf C}(\Pm)\right\}
\xrightarrow[n\rightarrow\infty]{d}
\mathbf N_{\mathbf B,\mathbf C}(\Pm)
\mathbf B\Wm_{\!\Pm}\mathbf C
\mathbf N_{\mathbf B,\mathbf C}(\Pm).
\end{equation}
For fixed $\mathbf u\in\mathbb R^{1\times r}$ and $\mathbf v\in\mathbb R^{r\times1}$,
\begin{equation*}
\phi(\Pm)=\mathbf u\mathbf N_{\mathbf B,\mathbf C}(\Pm)\mathbf v
\end{equation*}
has limiting variable
\begin{equation*}
\mathbf u\mathbf N_{\mathbf B,\mathbf C}(\Pm)
\mathbf B\Wm_{\!\Pm}\mathbf C
\mathbf N_{\mathbf B,\mathbf C}(\Pm)\mathbf v.
\end{equation*}
\end{proposition}

\begin{proof}
Let $\Psi(\mathbf M)=\mathbf I_r-\mathbf B\mathbf M\mathbf C$. Then
\begin{equation*}
\Psi'_{\Pm}(\mathbf H)=-\mathbf B\mathbf H\mathbf C.
\end{equation*}
Since $\rho(\mathbf B\Pm\mathbf C)<1$, the matrix $\Psi(\Pm)$ is invertible and its inverse is $\mathbf N_{\mathbf B,\mathbf C}(\Pm)$. The inverse rule gives
\begin{equation*}
(\Psi^{-1})'_{\Pm}(\mathbf H)
=-\Psi(\Pm)^{-1}\Psi'_{\Pm}(\mathbf H)\Psi(\Pm)^{-1},
\end{equation*}
which is exactly \eqref{eq:resolvent_derivative}. The convergence \eqref{eq:resolvent_limit} follows by applying Theorem \ref{theo:matrix_space_delta} to this differentiable inverse map. The scalar formula is obtained by fixed left and right multiplication.
\end{proof}

The stationary distribution is another inverse-type functional, but it is attached to the full chain rather than to a transient block. Let $\mathbf A=\mathbf 1_s\mathbf 1_s^\top$. For $\Pm\in\mathcal P_{\rm ir}$,
\begin{equation}
\label{eq:stationary_distribution}
\boldsymbol\pi=
\mathbf 1_s^\top(\mathbf I_s-\Pm+\mathbf A)^{-1}.
\end{equation}

\begin{proposition}
\label{prop:stationary_vector}
The plug-in estimator $\widehat{\boldsymbol\pi}(n)=\mathbf 1_s^\top(\mathbf I_s-\widehat\Pm(n)+\mathbf A)^{-1}$ is strongly consistent and
\begin{equation}
\label{asympt_law_stationary}
\sqrt n\left\{\widehat{\boldsymbol\pi}(n)-\boldsymbol\pi\right\}
\xrightarrow[n\rightarrow\infty]{d}
\Wm_{\!\boldsymbol\pi}:=
\boldsymbol\pi\Wm_{\!\Pm}(\mathbf I_s-\Pm+\mathbf A)^{-1}.
\end{equation}
Moreover, if
\begin{equation*}
\mathbf M_{\boldsymbol\pi^\top}=
\boldsymbol\pi\otimes(\mathbf I_s-\Pm^\top+\mathbf A)^{-1},
\end{equation*}
then
\begin{equation}
\label{vec_pi}
\Wm_{\!\boldsymbol\pi^\top}:=\vecc(\Wm_{\!\boldsymbol\pi}^\top)=
\mathbf M_{\boldsymbol\pi^\top}\Wm_{\!\pv}
\sim
\mathcal N_s(\mathbf 0_s,\Sm_{\boldsymbol\pi^\top}),
\end{equation}
where
\begin{equation}
\label{eq:cov_stationary}
\Sm_{\boldsymbol\pi^\top}=
\mathbf M_{\boldsymbol\pi^\top}\Sm_{\pv}\mathbf M_{\boldsymbol\pi^\top}^\top.
\end{equation}
\end{proposition}

\begin{proof}
For an irreducible finite chain, $\mathbf I_s-\Pm+\mathbf A$ is non-singular and
\begin{equation*}
\boldsymbol\pi=\mathbf{1}_s^\top(\mathbf I_s-\Pm+\mathbf A)^{-1}.
\end{equation*}
Set $\Psi(\mathbf M)=\mathbf I_s-\mathbf M+\mathbf A$. Then $\Psi'_{\Pm}(\mathbf H)=-\mathbf H$. By the inverse rule,
\begin{equation*}
(\Psi^{-1})'_{\Pm}(\mathbf H)
=(\mathbf I_s-\Pm+\mathbf A)^{-1}\mathbf H(\mathbf I_s-\Pm+\mathbf A)^{-1}.
\end{equation*}
Multiplication by $\mathbf{1}_s^\top$ on the left gives the derivative of the stationary vector map. Since $\mathbf{1}_s^\top(\mathbf I_s-\Pm+\mathbf A)^{-1}=\boldsymbol\pi$, insertion of $\Wm_{\!\Pm}$ yields \eqref{asympt_law_stationary}. Finally, Lemma \ref{lem: vec_propereties} gives the row-wise vector form and hence the covariance matrix in \eqref{eq:cov_stationary}.
\end{proof}


\subsection{Further Markov characteristics}
\label{subsec:further_characteristics}

The same calculus also gives compact asymptotic laws for quantities that summarize the long-run behaviour of a finite Markov chain and are not primarily reliability indicators. We record three examples. The first is the asymptotic variance in the central limit theorem for an additive functional. The second is Kemeny's constant, which is a global characteristic of the transition mechanism. The third is the entropy rate, a basic information-theoretic characteristic of the chain.

Let $\mathbf g=(g(1),\ldots,g(s))^\top$ be fixed and let $\mu_g=\boldsymbol\pi\mathbf g$. Put
\begin{equation*}
\mathbf g_0=\mathbf g-\mu_g\mathbf 1_s,
\qquad
\mathbf D_{\mathbf g_0}=\diag(g_0(1),\ldots,g_0(s)),
\qquad
\mathbf R_{\pi}=(\mathbf I_s-\Pm+\mathbf 1_s\boldsymbol\pi)^{-1}.
\end{equation*}
Let $\mathbf h_g=\mathbf R_{\pi}\mathbf g_0$. The asymptotic variance in the Markov chain central limit theorem for $\sum_{m=0}^{n-1}\{g(X_m)-\mu_g\}$ is
\begin{equation}
\label{eq:additive_variance_functional}
\phi_g(\Pm)=2\boldsymbol\pi\mathbf D_{\mathbf g_0}\mathbf h_g-\boldsymbol\pi\mathbf D_{\mathbf g_0}\mathbf g_0.
\end{equation}
This is the usual Poisson-equation representation of the variance constant, written in a form suited to differentiation with respect to $\Pm$. If
\begin{equation*}
\mathbf Z_{\Pm}=
(\mathbf I_s-\Pm+\mathbf 1_s\boldsymbol\pi)^{-1}
-
\mathbf 1_s\boldsymbol\pi,
\end{equation*}
then $\mathbf h_g=\mathbf Z_{\Pm}\mathbf g_0$ and
\begin{equation*}
\phi_g(\Pm)
=
\boldsymbol\pi(\mathbf g_0^{\circ 2})
+
2\boldsymbol\pi\mathbf D_{\mathbf g_0}
(\mathbf Z_{\Pm}-\mathbf I_s)\mathbf g_0.
\end{equation*}
This is the standard variance constant in the central limit theorem for additive functionals of finite Markov chains.

\begin{proposition}
\label{prop:further_markov_characteristics}
Let $\Pm\in\mathcal P_{\rm ir}$. For the additive-functional variance $\phi_g$ in \eqref{eq:additive_variance_functional},
\begin{equation}
\label{eq:additive_variance_limit}
\sqrt n\{\phi_g(\widehat\Pm(n))-\phi_g(\Pm)\}
\xrightarrow[n\rightarrow\infty]{d}
\phi'_{g,\Pm}(\Wm_{\!\Pm}),
\end{equation}
where, for $\mathbf H\in\mathcal T_{\Pm}$,
\begin{equation}
\label{eq:additive_variance_derivative}
\phi'_{g,\Pm}(\mathbf H)=
\boldsymbol\pi\mathbf H\mathbf R_{\pi}
\left(2\mathbf D_{\mathbf g_0}\mathbf h_g-\mathbf D_{\mathbf g_0}\mathbf g_0\right)
+2\boldsymbol\pi\mathbf D_{\mathbf g_0}\mathbf R_{\pi}\mathbf H\mathbf h_g.
\end{equation}

The Kemeny constant
\begin{equation*}
K(\Pm)=\tr\{(\mathbf I_s-\Pm+\mathbf 1_s\boldsymbol\pi)^{-1}\}
\end{equation*}
satisfies
\begin{equation}
\label{eq:kemeny_limit}
\sqrt n\{K(\widehat\Pm(n))-K(\Pm)\}
\xrightarrow[n\rightarrow\infty]{d}
\tr(\mathbf R_{\pi}\Wm_{\!\Pm}\mathbf R_{\pi}).
\end{equation}

Assume finally that all transition probabilities are positive. The entropy rate
\begin{equation*}
\mathcal H(\Pm)=-\sum_{i=1}^s\pi_i\sum_{j=1}^s p_{ij}\log p_{ij}
\end{equation*}
satisfies
\begin{equation}
\label{eq:entropy_limit}
\sqrt n\{\mathcal H(\widehat\Pm(n))-\mathcal H(\Pm)\}
\xrightarrow[n\rightarrow\infty]{d}
\boldsymbol\pi\Wm_{\!\Pm}\mathbf R_{\pi}\mathbf r_{\Pm}
-\sum_{i=1}^s\pi_i\sum_{j=1}^s W_{ij}\log p_{ij},
\end{equation}
where $\mathbf r_{\Pm}\in\mathbb R^{s}$ has entries
\begin{equation*}
(\mathbf r_{\Pm})_i=-\sum_{j=1}^s p_{ij}\log p_{ij}.
\end{equation*}
\end{proposition}

\begin{proof}
The matrix $\mathbf h_g$ is the solution of the Poisson equation $(\mathbf I_s-\Pm)\mathbf h_g=\mathbf g_0$ satisfying $\boldsymbol\pi\mathbf h_g=0$. Hence \eqref{eq:additive_variance_functional} is the standard variance constant of the additive-functional central limit theorem. The derivative of the stationary vector is $\boldsymbol\pi\mathbf H\mathbf R_{\pi}$. Since $\mathbf H\mathbf 1_s=\mathbf 0_s$, the derivative of $\mu_g$ is $\boldsymbol\pi\mathbf H\mathbf R_{\pi}\mathbf g=\boldsymbol\pi\mathbf H\mathbf h_g$. The terms obtained by differentiating $\mathbf g_0$ disappear after multiplication by $\boldsymbol\pi$, because $\boldsymbol\pi\mathbf g_0=0$ and $\boldsymbol\pi\mathbf h_g=0$. Differentiating $\mathbf h_g=\mathbf R_{\pi}\mathbf g_0$ and using the same cancellations gives \eqref{eq:additive_variance_derivative}. The limit \eqref{eq:additive_variance_limit} follows from Theorem \ref{theo:matrix_space_delta}.

For $K$, the derivative of $\mathbf R_{\pi}$ in the direction $\mathbf H$ is $\mathbf R_{\pi}\mathbf H\mathbf R_{\pi}$ when evaluated on $\mathcal T_{\Pm}$; the term coming from the derivative of $\boldsymbol\pi$ has zero trace because $\mathbf R_{\pi}\mathbf 1_s=\mathbf 1_s$ and $\boldsymbol\pi\mathbf H\mathbf 1_s=0$. This proves \eqref{eq:kemeny_limit}.

For the entropy rate, write $\mathcal H(\Pm)=\boldsymbol\pi\mathbf r_{\Pm}$. The derivative of $\boldsymbol\pi$ gives the first term in \eqref{eq:entropy_limit}. Since every row of $\mathbf H$ sums to zero, differentiating $-\sum_j p_{ij}\log p_{ij}$ gives $-\sum_j h_{ij}\log p_{ij}$, which yields the second term. The positivity assumption guarantees differentiability of the logarithm at all entries.
\end{proof}


\section{Reliability functionals and Markov characteristics}
\label{sec: reliability}

The preceding propositions give a common approach to the asymptotic distributions of many quantities used to describe the performance of a Markov system. We present here the probabilistic definitions of the reliability quantities of interest.  The asymptotic derivations, however, no longer need to be carried out separately for each indicator. Once a characteristic is written as a power, a product, a resolvent or a scalar linear form, its plug-in MLE is handled by the same limiting Gaussian matrix $\Wm_{\!\Pm}$. This produces the limit law directly in probabilistic form. The covariance matrices required for confidence intervals, joint confidence regions and tests are then obtained from the operators of Section \ref{sec: matrix_calculus}. Thus the results of this section should be understood as a list of explicit instances of the general calculus rather than as a collection of independent delta-method calculations.

Let the state space be partitioned into $U=\{1,\ldots,r\}$ and $D=\{r+1,\ldots,s\}$, interpreted as working and failed states. Let $\mathbf E_U\in\mathbb R^{s\times r}$ and $\mathbf E_D\in\mathbb R^{s\times(s-r)}$ be the corresponding selection matrices. Define
\begin{equation*}
\Pm_{UU}=\mathbf E_U^\top\Pm\mathbf E_U,
\quad
\Pm_{UD}=\mathbf E_U^\top\Pm\mathbf E_D,
\quad
\Pm_{DD}=\mathbf E_D^\top\Pm\mathbf E_D,
\end{equation*}
\begin{equation*}
\mathbf a_U=\mathbf a\mathbf E_U,
\quad
\mathbf a_D=\mathbf a\mathbf E_D,
\quad
\mathbf e_U=\mathbf E_U\mathbf 1_r,
\quad
\mathbf e_D=\mathbf E_D\mathbf 1_{s-r},
\quad
\mathbf{J}_U=\mathbf E_U\mathbf E_U^\top.
\end{equation*}
The initial law $\mathbf a$ is assumed known, as it cannot be estimated consistently from a single trajectory.

\subsection{Probabilistic definitions}
\label{subsec:prob_definitions}

The availability at time $k$ is
\begin{equation}
\label{eq: availability}
A_k=\mathbb P(X_k\in U)=\mathbf a\Pm^k\mathbf e_U.
\end{equation}
The availability curve on $\{0,\ldots,m\}$ is the vector $(A_0,\ldots,A_m)^\top$, which is the curve \eqref{eq:general_curve} with $\mathbf c=\mathbf e_U$.

Let
\begin{equation*}
T_D=\inf\{n\geq0:X_n\in D\}.
\end{equation*}
The reliability function, the failure-time probability and the mean time to failure are
\begin{align}
R_k&=\mathbb P(T_D>k)=\mathbf a_U\Pm_{UU}^k\mathbf 1_r,
\label{eq:reliability_curve}\\
f_k&=\mathbb P(T_D=k)=\mathbf a_U\Pm_{UU}^{k-1}\Pm_{UD}\mathbf 1_{s-r},\qquad k\geq1,
\label{eq:failure_density}\\
\operatorname{MTTF}&=\mathbb E(T_D)=
\mathbf a_U(\mathbf I_r-\Pm_{UU})^{-1}\mathbf 1_r,
\label{eq: MTTF}
\end{align}
whenever $\rho(\Pm_{UU})<1$. In an irreducible chain with $U$ and $D$ nonempty, this condition holds.

Similarly, with
\begin{equation*}
T_U=\inf\{n\geq0:X_n\in U\},
\end{equation*}
the down-state survival curve and the mean time to repair are
\begin{align}
Q_k&=\mathbb P(T_U>k)=\mathbf a_D\Pm_{DD}^k\mathbf 1_{s-r},
\label{eq:down_survival}\\
\operatorname{MTTR}&=\mathbb E(T_U)=
\mathbf a_D(\mathbf I_{s-r}-\Pm_{DD})^{-1}\mathbf 1_{s-r},
\label{eq: MTTR}
\end{align}
whenever $\rho(\Pm_{DD})<1$.

The discrete-time intensity of hitting $D$ at time $k\geq1$ is the discrete analogue of the rate of occurrence of failures. Related forms of this indicator have been studied for semi-Markov and hidden Markov renewal models in \cite{votsi2014,votsi2015,votsi2019}. It is defined here by
\begin{equation}
\label{eq:DTIHT}
r_k=\mathbb P(X_{k-1}\in U,X_k\in D)
=\mathbf a\Pm^{k-1}\mathbf{J}_U\Pm\mathbf e_D.
\end{equation}
The stationary availability is
\begin{equation}
\label{eq:stationary_availability}
A_\infty=\boldsymbol\pi\mathbf e_U.
\end{equation}

\subsection{Asymptotic representations}
\label{subsec:reliability_asymptotics}

The MLEs are obtained by replacing $\Pm$ by $\widehat\Pm(n)$ in the preceding formulas. Element-wise asymptotic expressions for several reliability indicators were derived in \cite{sadek_limnios}. The point here is different: each indicator is first written as a simple matrix functional, and the limiting random variable is obtained by inserting the same limiting Gaussian matrix $\Wm_{\!\Pm}$ into the corresponding derivative. The formulas below are therefore displayed in their natural probabilistic form. When a numerical covariance matrix is required for confidence intervals, confidence regions or tests, it is recovered from the vector operators of Section \ref{sec: matrix_calculus}. This gives a common approach to marginal and joint inference for several times and several characteristics.

\begin{proposition}
\label{prop:availability_curve}
For every fixed $m\geq0$,
\begin{equation}
\label{eq:availability_curve_limit}
\sqrt n\left\{(\widehat A_0(n),\ldots,\widehat A_m(n))^\top-(A_0,\ldots,A_m)^\top\right\}
\xrightarrow[n\rightarrow\infty]{d}
\mathbf Z_A,
\end{equation}
where
\begin{equation*}
(\mathbf Z_A)_k=\mathbf a\Wm_{\!\Pm^k}\mathbf e_U,
\qquad k=0,
\ldots,m.
\end{equation*}
Its covariance matrix is given by \eqref{eq:curve_cov} with $\mathbf c=\mathbf e_U$. In particular, the cumulative availability
\begin{equation*}
C_m=\sum_{k=0}^m A_k
\end{equation*}
has limiting variable
\begin{equation*}
\sum_{k=0}^m\mathbf a\Wm_{\!\Pm^k}\mathbf e_U.
\end{equation*}
\end{proposition}

\begin{proof}
The identity $A_k=\mathbf a\Pm^k\mathbf e_U$ shows that the vector $(A_0,\ldots,A_m)^\top$ is the Markov curve \eqref{eq:general_curve} with $\mathbf c=\mathbf e_U$. Proposition \ref{prop:markov_curves} gives the joint limit of all ordinates. The cumulative availability is the image of this vector under the linear map $(x_0,\ldots,x_m)\mapsto\sum_{k=0}^m x_k$, and therefore its limiting variable is the corresponding sum of the components of $\mathbf Z_A$.
\end{proof}

\begin{proposition}
\label{prop:reliability_failure}
Let $\Sm_{\pv_{UU}}$ and $\Sm_{\pv_{DD}}$ be the covariance submatrices of $\Sm_{\pv}$ corresponding respectively to $\vecc(\Wm_{\!\Pm_{UU}}^\top)$ and $\vecc(\Wm_{\!\Pm_{DD}}^\top)$. For $k\geq0$,
\begin{equation}
\label{eq:reliability_limit}
\sqrt n\left\{\widehat R_k(n)-R_k\right\}
\xrightarrow[n\rightarrow\infty]{d}
\mathbf a_U\Wm_{\!\Pm_{UU}^k}\mathbf 1_r,
\end{equation}
where $\Wm_{\!\Pm_{UU}^0}=\mathbf 0$. If
\begin{equation*}
\mathbf M^{UU}_k=
\sum_{i=0}^{k-1}\Pm_{UU}^i\otimes(\Pm_{UU}^{k-1-i})^\top,
\qquad k\geq1,
\end{equation*}
with $\mathbf M^{UU}_0=\mathbf 0_{r^2\times r^2}$, then
\begin{equation}
\label{eq:reliability_var}
\VV(\mathbf a_U\Wm_{\!\Pm_{UU}^k}\mathbf 1_r)=
(\mathbf a_U\otimes\mathbf 1_r^\top)
\mathbf M^{UU}_k\Sm_{\pv_{UU}}(\mathbf M^{UU}_k)^\top
(\mathbf a_U^\top\otimes\mathbf 1_r).
\end{equation}
For $k\geq1$,
\begin{equation}
\label{eq:failure_density_limit}
\sqrt n\left\{\widehat f_k(n)-f_k\right\}
\xrightarrow[n\rightarrow\infty]{d}
\mathbf a_U\Big(\Wm_{\!\Pm_{UU}^{k-1}}\Pm_{UD}
+\Pm_{UU}^{k-1}\Wm_{\!\Pm_{UD}}\Big)\mathbf 1_{s-r}.
\end{equation}
Furthermore,
\begin{equation}
\label{eq:down_survival_limit}
\sqrt n\left\{\widehat Q_k(n)-Q_k\right\}
\xrightarrow[n\rightarrow\infty]{d}
\mathbf a_D\Wm_{\!\Pm_{DD}^k}\mathbf 1_{s-r},
\end{equation}
and the maintainability $M_k=1-Q_k$ has limiting variable
\begin{equation*}
-\mathbf a_D\Wm_{\!\Pm_{DD}^k}\mathbf 1_{s-r}.
\end{equation*}
\end{proposition}

\begin{proof}
The reliability function is obtained by evolving the chain inside the up-state block before hitting $D$. Thus $R_k$ is the scalar linear form of the power $\Pm_{UU}^k$, and Proposition \ref{prop:powers}, applied to the restricted block, gives \eqref{eq:reliability_limit}. The variance formula is the corresponding scalar linear form covariance.

For the failure probability $f_k$, the functional is the product $\Pm_{UU}^{k-1}\Pm_{UD}$ between the survival part in $U$ and the final jump to $D$. Differentiating this product gives the two terms in \eqref{eq:failure_density_limit}. The down-state survival formula is the same argument applied to the down-state block $\Pm_{DD}$. Finally, maintainability is $1-Q_k$, so its first-order term is the negative of the first-order term of $Q_k$.
\end{proof}

\begin{proposition}
\label{prop:mean_times}
Let $\mathbf N_U=(\mathbf I_r-\Pm_{UU})^{-1}$ and $\mathbf N_D=(\mathbf I_{s-r}-\Pm_{DD})^{-1}$. Then
\begin{equation}
\label{eq:mttf_limit}
\sqrt n\left\{\widehat{\operatorname{MTTF}}(n)-\operatorname{MTTF}\right\}
\xrightarrow[n\rightarrow\infty]{d}
\mathbf a_U\mathbf N_U\Wm_{\!\Pm_{UU}}\mathbf N_U\mathbf 1_r,
\end{equation}
with variance
\begin{equation}\label{var:mttf}
\sigma^2_{\operatorname{MTTF}}=
\mathbf L_U\Sm_{\pv_{UU}}\mathbf L_U^\top,
\qquad
\mathbf L_U=
\mathbf a_U\mathbf N_U\otimes \mathbf 1_r^\top\mathbf N_U^\top.
\end{equation}
Similarly,
\begin{equation}
\label{eq:mttr_limit}
\sqrt n\left\{\widehat{\operatorname{MTTR}}(n)-\operatorname{MTTR}\right\}
\xrightarrow[n\rightarrow\infty]{d}
\mathbf a_D\mathbf N_D\Wm_{\!\Pm_{DD}}\mathbf N_D\mathbf 1_{s-r},
\end{equation}
with variance
\begin{equation}
\label{var:mttr}
\sigma^2_{\operatorname{MTTR}}=
\mathbf L_D\Sm_{\pv_{DD}}\mathbf L_D^\top,
\qquad
\mathbf L_D=
\mathbf a_D\mathbf N_D\otimes \mathbf 1_{s-r}^\top\mathbf N_D^\top.
\end{equation}
\end{proposition}

\begin{proof}
The MTTF is a scalar linear form of the restricted resolvent $(\mathbf I_r-\Pm_{UU})^{-1}$. Proposition \ref{prop:resolvent}, with the selection matrices corresponding to the up-state block, gives
\begin{equation*}
\{(\mathbf I_r-\Pm_{UU})^{-1}\}'(\Wm_{\!\Pm_{UU}})
=
\mathbf N_U\Wm_{\!\Pm_{UU}}\mathbf N_U.
\end{equation*}
Multiplication by $\mathbf a_U$ and $\mathbf{1}_r$ gives \eqref{eq:mttf_limit}. The displayed variance follows from the identity
\begin{equation*}
\mathbf a_U\mathbf N_U\Wm_{\!\Pm_{UU}}\mathbf N_U\mathbf{1}_r
=
\mathbf L_U\vecc(\Wm_{\!\Pm_{UU}}^\top).
\end{equation*}
The MTTR formula is identical with the down-state block in place of the up-state block.
\end{proof}

\begin{proposition}
\label{prop:intensity_stationary_availability}
For $k\geq1$,
\begin{equation}
\label{eq:dtiht_limit}
\sqrt n\left\{\widehat r_k(n)-r_k\right\}
\xrightarrow[n\rightarrow\infty]{d}
\mathbf a\Big(\Wm_{\!\Pm^{k-1}}\mathbf{J}_U\Pm+
\Pm^{k-1}\mathbf{J}_U\Wm_{\!\Pm}\Big)\mathbf e_D,
\end{equation}
with the convention $\Wm_{\!\Pm^0}=\mathbf 0$. Moreover,
\begin{equation}
\label{eq:stationary_availability_limit}
\sqrt n\left\{\widehat A_\infty(n)-A_\infty\right\}
\xrightarrow[n\rightarrow\infty]{d}
\Wm_{\!\boldsymbol\pi}\mathbf e_U,
\end{equation}
where $\Wm_{\!\boldsymbol\pi}$ is given by \eqref{asympt_law_stationary}.
\end{proposition}

\begin{proof}
The intensity $r_k$ is the scalar linear form of the product $\Pm^{k-1}\mathbf{J}_U\Pm$. The first factor describes the distribution at time $k-1$, the matrix $\mathbf{J}_U$ keeps only paths that are in the up set at that time, and the last factor gives the jump to $D$. Differentiating the product gives a contribution from the first factor and a contribution from the last factor, namely the two terms in \eqref{eq:dtiht_limit}. When $k=1$, the factor $\Pm^0$ is deterministic and its derivative is zero, so the convention $\Wm_{\!\Pm^0}=0$ gives the same formula. The stationary availability is the scalar linear form of the stationary vector; Proposition \ref{prop:stationary_vector} gives \eqref{eq:stationary_availability_limit}.
\end{proof}


\section{Numerical examples}
\label{sec: numerical}

In this section, we present numerical results illustrating the computational form of the preceding covariance formulas and the use of the first-order and curvature-refined approximations. All numerical values in this section are computed for the same transition matrix, and the Monte Carlo experiments are used only to compare the distribution of the plug-in estimators with the approximations obtained from the Gaussian matrix calculus. The calculations were made in R (version 4.0.1). The PC used is a DELL Precision Mobile 3591, with memory 32 GB and processor Intel(R) Core(TM) Ultra 7 165H (1.40 GHz).

\subsection{Computational comparison of covariance representations}

For $k\geq1$, consider the row-wise vectorization of $\sqrt n\{\widehat\Pm^k(n)-\Pm^k\}$. Its covariance matrix is
\begin{equation}
\label{cov:matrix}
\Sm_{\pv^{(k)}}=
\mathbf M_k\Sm_{\pv}\mathbf M_k^\top.
\end{equation}
The matrix $\mathbf M_k$ is obtained once from \eqref{Mk}. In the element-wise approach, the entries of the Jacobian of the map $\Pm\mapsto\Pm^k$ are computed by
\begin{equation}
\label{eq: covariance_Fn}
\frac{\partial p^{(k)}_{ij}}{\partial p_{uv}}=
\delta_{iu}p^{k-1}_{vj}+
\sum_{m=1}^{k-2}p^{m}_{iu}p^{k-1-m}_{vj}
+p^{k-1}_{iu}\delta_{vj},
\end{equation}
for $i,j,u,v\in S$. The covariance is then obtained by combining this Jacobian with $\Sm_{\pv}$.

We generate stochastic matrices corresponding to different numbers of states using the random stochastic-matrix generator of \cite{RMAT} and, for each matrix, compute the covariance matrix of $\sqrt n\left\{\widehat\pv^{(k)}(n)-\pv^{(k)}\right\}$ for $k=2$ by the element-wise and matrix-based approaches. For each fixed number of states $s$, ten stochastic matrices are generated and the mean computational time is recorded.

\begin{figure}[!htbp]
\centering
\includegraphics[scale=0.65]{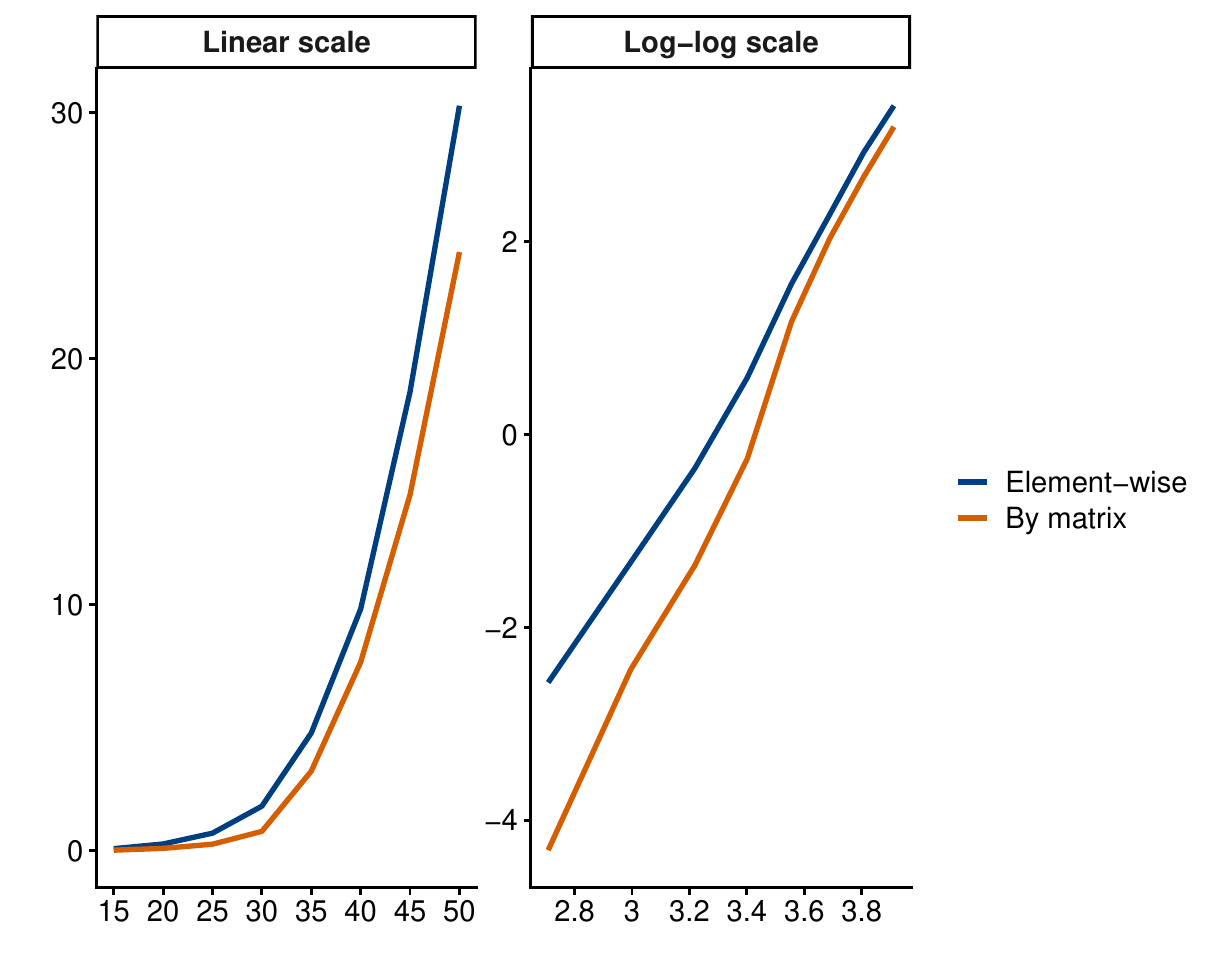}
\caption{Computation times versus number of states.}
\label{comput_times}
\end{figure}

According to the left panel of Fig.\ref{comput_times}, the element-wise approach is more time-consuming than the matrix-based approach, and the computational gap becomes more pronounced as the dimension of the state space increases. In the right panel of Fig.\ref{comput_times}, computation times are displayed on a logarithmic scale. This comparison concerns the construction of the covariance matrix itself; it is independent of the subsequent scalar, vector or matrix functional to which the covariance is applied.

\subsection{Pointwise inference for availability}

Consider a three-state Markov chain with initial law $\mathbf a$ and transition matrix $\Pm$ given by
\begin{equation*}
\mathbf a=
\begin{pmatrix}
0.5&0.5&0
\end{pmatrix},
\qquad
\Pm=
\begin{pmatrix}
0.3&0.7&0\\
0.6&0.1&0.3\\
0.5&0&0.5
\end{pmatrix}.
\end{equation*}
The up states are $U=\{1,2\}$ and the down state is $D=\{3\}$. For this model the mean time to failure is
\begin{equation*}
\operatorname{MTTF}=6.9048.
\end{equation*}
The limiting availability is close to $0.7921$, and already for moderate $k$ the values of $A_k$ are close to this level.

We first compute the theoretical availability curve and compare it with the plug-in MLE obtained from a simulated trajectory. The pointwise confidence intervals in Fig.\ref{availability2} are based on \eqref{eq:availability_curve_limit}.

\begin{figure}[!htbp]
\centering
\includegraphics[scale=0.65]{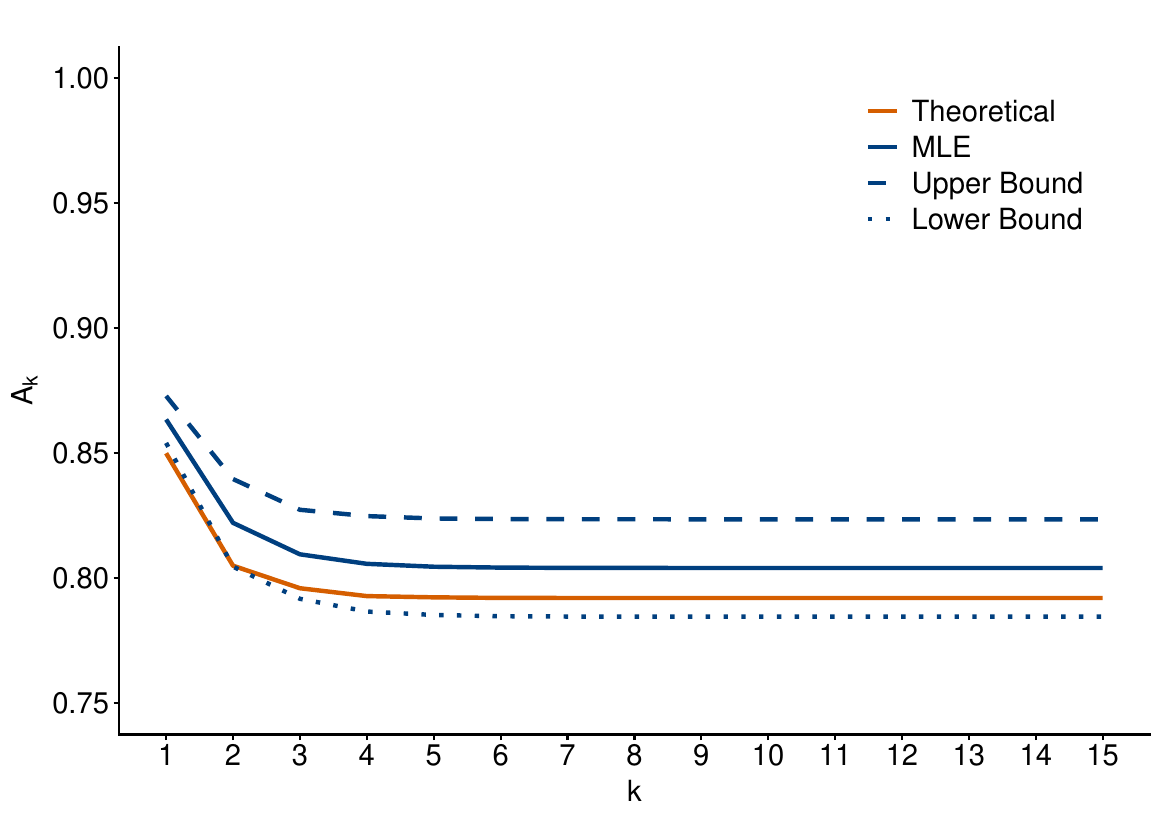}
\caption{Availability, MLE and 95\% confidence interval.}
\label{availability2}
\end{figure}

The following table gives representative numerical values for the availability estimator at three time instants, $k=5,10,20$. For $n=500,2000,10000$, the empirical quantities are computed from $40000$, $30000$ and $10000$ independent trajectories, respectively. The column ``Empirical interval'' is the Monte Carlo $95\%$ interval of the plug-in estimator $\widehat A_k(n)$. The column ``First-order interval'' is the interval
\begin{equation*}
A_k\pm 1.96\,\frac{\sigma_{A_k}}{\sqrt n},
\end{equation*}
where $\sigma_{A_k}^2$ is obtained from the matrix formula in Proposition \ref{prop:availability_curve}. The table is on the original availability scale.

\begin{table}[!htbp]
\centering
\small
\setlength{\tabcolsep}{4pt}
\caption{Availability values and first-order approximation on the original scale.}
\label{tab:availability_values}
\begin{tabular}{c|c|c|c|c|c}
$n$ & $k$ & $A_k$ & Empirical mean & Empirical interval & First-order interval\\ \hline
$500$ & $5$  & $0.7924$ & $0.7925$ & $[0.7423,0.8400]$ & $[0.7434,0.8413]$\\
$500$ & $10$ & $0.7921$ & $0.7920$ & $[0.7420,0.8398]$ & $[0.7429,0.8412]$\\
$500$ & $20$ & $0.7921$ & $0.7920$ & $[0.7420,0.8398]$ & $[0.7429,0.8412]$\\ \hline
$2000$ & $5$  & $0.7924$ & $0.7924$ & $[0.7679,0.8164]$ & $[0.7679,0.8168]$\\
$2000$ & $10$ & $0.7921$ & $0.7921$ & $[0.7675,0.8161]$ & $[0.7675,0.8166]$\\
$2000$ & $20$ & $0.7921$ & $0.7921$ & $[0.7675,0.8161]$ & $[0.7675,0.8166]$\\ \hline
$10000$ & $5$  & $0.7924$ & $0.7923$ & $[0.7812,0.8032]$ & $[0.7814,0.8033]$\\
$10000$ & $10$ & $0.7921$ & $0.7921$ & $[0.7809,0.8029]$ & $[0.7811,0.8031]$\\
$10000$ & $20$ & $0.7921$ & $0.7921$ & $[0.7809,0.8029]$ & $[0.7811,0.8031]$
\end{tabular}
\end{table}

Table \ref{tab:availability_values} shows that the matrix covariance representation gives the correct scale of the estimator without any coordinate-wise differentiation. The contraction of the intervals from $n=500$ to $n=10000$ also illustrates the usual $n^{-1/2}$ behavior. The values at $k=10$ and $k=20$ are close because the availability curve has essentially reached its stationary level for this transition matrix.

\subsection{Availability bands and MTTF first-order asymptotics}

We continue with the three-state Markov chain with initial law $\mathbf a$ and transition matrix $\Pm$ given by
\begin{equation*}
\mathbf a=
\begin{pmatrix}
  0.5 & 0.5 & 0
\end{pmatrix},
\qquad
\Pm=
\begin{pmatrix}
  0.3 & 0.7 & 0 \\
  0.6 & 0.1 & 0.3  \\
  0.5 & 0 & 0.5
\end{pmatrix}.
\end{equation*}
The up and down sets are $U=\{1,2\}$ and $D=\{3\}$, respectively. We write $\mathbf e_U=\mathbf 1_{r,s}$ for the column vector with entries equal to one on the up states and zero otherwise.

The first illustration concerns the joint distribution of a finite-dimensional Markov curve. Put
\begin{equation*}
\mathbf A_{1:20}=(A_1,\ldots,A_{20})^\top,
\qquad
A_k=\mathbf a\Pm^k\mathbf e_U.
\end{equation*}
For each $k$, the Fr\'echet derivative of $A_k$ at $\Pm$ is
\begin{equation*}
D A_k(\Pm)[\Hm]
=
\mathbf a
\left(\sum_{\ell=0}^{k-1}\Pm^\ell\Hm\Pm^{k-1-\ell}\right)
\mathbf e_U.
\end{equation*}
Thus the covariance matrix of the limiting vector
\begin{equation*}
\sqrt n\left\{\widehat{\mathbf A}_{1:20}(n)-\mathbf A_{1:20}\right\}
\end{equation*}
is obtained by evaluating the same matrix derivative simultaneously at the times $1,\ldots,20$. This is a useful point of the representation: the scalar asymptotic variance and the full covariance matrix of the curve are produced by the same linear map, without separate coordinate-wise differentiations.

This joint covariance gives simultaneous confidence bands. Let $\Gamma_A$ be the covariance matrix of the limiting curve, let $\sigma_k^2=(\Gamma_A)_{kk}$, and let $\mathbf G=(G_1,\ldots,G_{20})^\top$ be a centered Gaussian vector with covariance matrix $\Gamma_A$. The simultaneous Gaussian critical value is the $0.95$-quantile of
\begin{equation*}
\max_{1\leq k\leq20}\left|G_k/\sigma_k\right|.
\end{equation*}
For the present model, the population value is $2.196$, whereas the Bonferroni critical value for twenty coordinates is $3.023$. For an observed trajectory, $\Gamma_A$ and the associated Gaussian quantile are replaced by their plug-in estimates, as described in Remark~\ref{rem:plug_in_inference}. Hence the feasible joint Gaussian band uses the estimated dependence structure of the curve and is substantially narrower than its Bonferroni counterpart.

\begin{table}[!htbp]
\centering
\small
\setlength{\tabcolsep}{4pt}
\caption{Critical values, simultaneous Gaussian coverage and average half-widths for the availability curve $A_1,\ldots,A_{20}$.}
\label{tab:curve_bands}
\begin{tabular}{c|c|c|c|c|c}
Band & Critical value & Coverage & $n=500$ & $n=2000$ & $n=10000$\\
\hline
Pointwise & $1.960$ & $0.914$ & $0.0481$ & $0.0241$ & $0.0108$\\
Joint Gaussian & $2.196$ & $0.950$ & $0.0539$ & $0.0270$ & $0.0121$\\
Bonferroni & $3.023$ & $0.995$ & $0.0742$ & $0.0371$ & $0.0166$
\end{tabular}
\end{table}

The column ``Coverage'' gives the simultaneous coverage under the limiting Gaussian law of the whole curve. The last three columns give the average half-width of the corresponding band on the original availability scale. Pointwise intervals have the smallest marginal width, but their simultaneous coverage is only $0.914$. The Bonferroni band is simultaneous but conservative. The joint Gaussian band attains the target simultaneous coverage and reduces the average half-width by about $27\%$ compared with Bonferroni, for each value of $n$.

Figure~\ref{fig:availability_curve_simultaneous} shows one trajectory of length $n=2000$. The bands are centered at the plug-in availability curve and are computed with the plug-in covariance matrix. The Bonferroni band is displayed for comparison. Both the joint Gaussian and Bonferroni bands are simultaneous for the same finite-dimensional curve, but the joint band is visibly narrower. The pointwise analogue is also displayed for reference.

\begin{figure}[!htbp]
\centering
\includegraphics[scale=0.65]{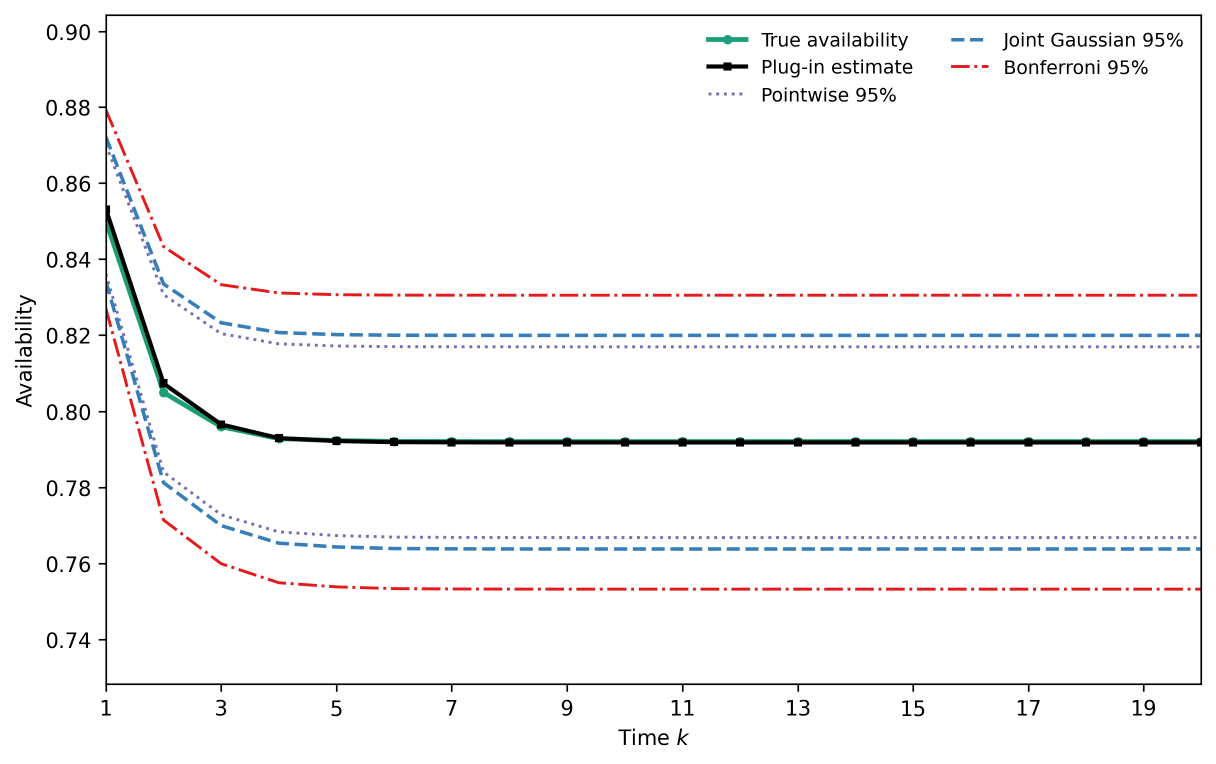}
\caption{Availability curve $A_1,\ldots,A_{20}$, plug-in estimator and confidence bands for $n=2000$.}
\label{fig:availability_curve_simultaneous}
\end{figure}

We next illustrate the first-order MTTF limit in \eqref{eq:mttf_limit}. For each simulated trajectory, the transition matrix is estimated from the transition counts and the plug-in estimator
\begin{equation*}
\widehat{\operatorname{MTTF}}(n)
=
\mathbf a_1
\left(\mathbf I_r-\widehat{\Pm}_{11}(n)\right)^{-1}
\mathbf 1_r
\end{equation*}
is computed. Figure~\ref{fig:mttf_normalized_error_clt} represents the empirical distribution of
\begin{equation*}
\sqrt n\left\{\widehat{\operatorname{MTTF}}(n)-\operatorname{MTTF}\right\}
\end{equation*}
for $n=50000$, based on $10^4$ independent trajectories. The superimposed curve is the limiting normal density with variance given by \eqref{var:mttf}; it is not fitted to the histogram. This normalized-error figure is centered near zero. On the original scale, the estimator is centered near the theoretical value
\begin{equation*}
\operatorname{MTTF}=6.9048,
\end{equation*}
and the finite-sample deformation induced by the inverse map is examined in the next subsection.

\begin{figure}[!htbp]
\centering
\includegraphics[scale=0.65]{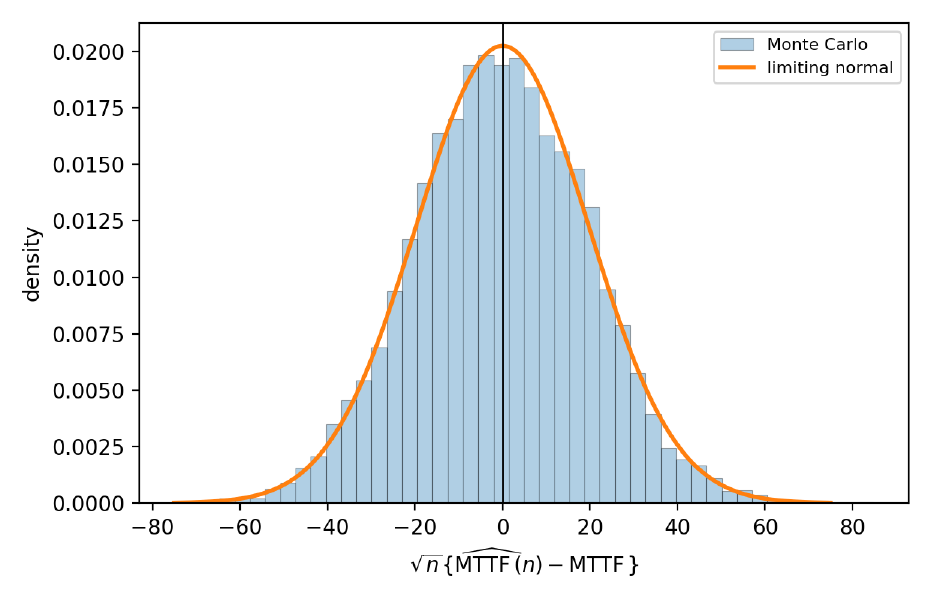}
\caption{Asymptotic distribution of the normalized MTTF error.}
\label{fig:mttf_normalized_error_clt}
\end{figure}

\subsection{Curvature-refined Gaussian approximation for the MTTF}

The same example is used to illustrate a second-order stochastic development on the original scale of the MTTF estimator. Let
\begin{equation*}
\operatorname{MTTF}(\Pm)
=
\mathbf a_1
\left(\mathbf I_r-\Pm_{11}\right)^{-1}
\mathbf 1_r
\end{equation*}
and put
\begin{equation*}
\mathbf N_U=
\left(\mathbf I_r-\Pm_{11}\right)^{-1}.
\end{equation*}
For a matrix direction $\Hm$, the first differential is
\begin{equation*}
L_{\Pm}(\Hm)
=
\mathbf a_1\mathbf N_U\Hm_{11}\mathbf N_U\mathbf 1_r.
\end{equation*}
The curvature term, that is, one half of the second differential evaluated at $(\Hm,\Hm)$, is
\begin{equation*}
Q_{\Pm}(\Hm)
=
\frac12\operatorname{MTTF}_{\Pm}^{(2)}(\Hm,\Hm)
=
\mathbf a_1\mathbf N_U\Hm_{11}\mathbf N_U
\Hm_{11}\mathbf N_U\mathbf 1_r.
\end{equation*}
Consequently, the first-order Gaussian approximation of the estimator is
\begin{equation*}
\operatorname{MTTF}(\Pm)
+
n^{-1/2}L_{\Pm}(\Wm_{\!\Pm}),
\end{equation*}
whereas the curvature-refined Gaussian approximation is
\begin{equation*}
\operatorname{MTTF}(\Pm)
+
n^{-1/2}L_{\Pm}(\Wm_{\!\Pm})
+
n^{-1}Q_{\Pm}(\Wm_{\!\Pm}).
\end{equation*}
The second-order term is not a correction of the transition-matrix estimator. It is the curvature term of the MTTF functional itself, and accounts for the asymmetry induced by the inverse map. For scalar confidence intervals based on the first-order theory, no Gaussian simulation is needed: one replaces $\Pm$ and $\Sm_{\pv}$ by their plug-in estimates in \eqref{var:mttf}. The simulations below are used only to compare the empirical distribution with the first-order and curvature-refined approximations.

Figure~\ref{fig:mttf_second_order_refinement} compares the empirical distribution of $\widehat{\operatorname{MTTF}}(n)$ with the first-order Gaussian and curvature-refined approximations, on the original MTTF scale. For $n=500$ and $n=2000$, the empirical curves are based on $5\times10^6$ independent Monte Carlo trajectories. For $n=10000$, a sample of $3\times10^5$ trajectories is used as a large-sample reference. The first-order approximation is normal and symmetric around $6.9048$. The curvature-refined approximation is obtained by simulating the limiting Gaussian matrix $\Wm_{\!\Pm}$ and adding the curvature term $n^{-1}Q_{\Pm}(\Wm_{\!\Pm})$. The curve labelled ``Second order'' in Figure~\ref{fig:mttf_second_order_refinement} denotes this curvature-refined approximation; no assertion of second-order distributional accuracy is made.

\begin{figure}[!htbp]
\centering
\includegraphics[width=\textwidth]{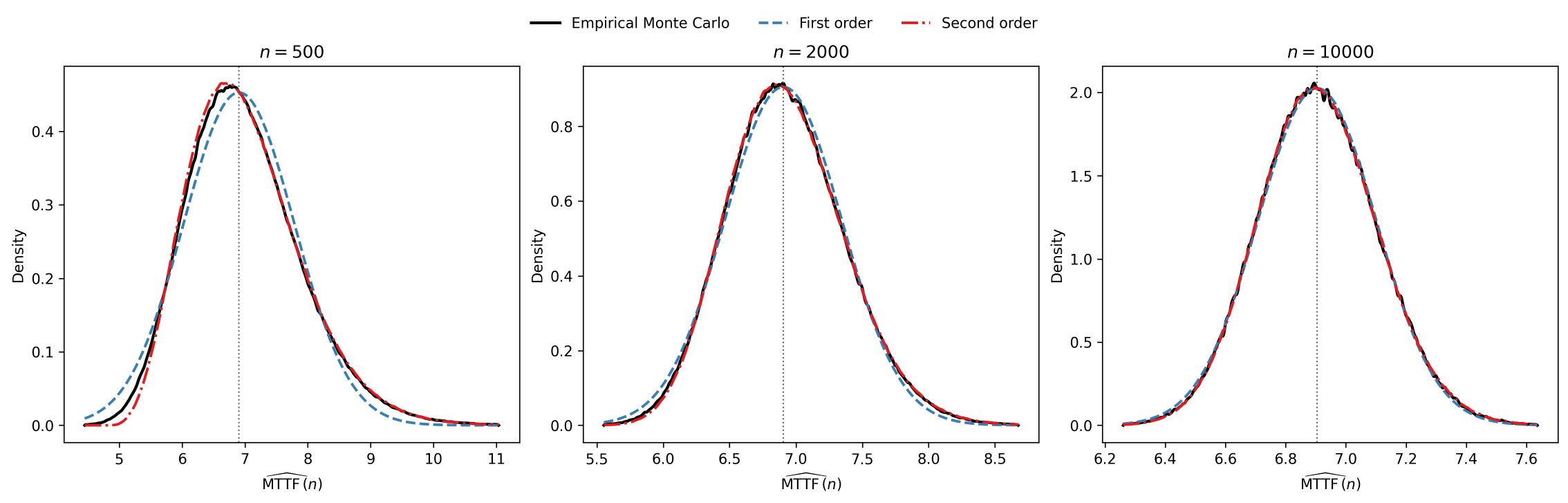}
\caption{Empirical distribution of $\widehat{\operatorname{MTTF}}(n)$ and first-order Gaussian and curvature-refined approximations, on the original MTTF scale.}
\label{fig:mttf_second_order_refinement}
\end{figure}

Table~\ref{tab:second_order_mttf} compares the empirical distribution of $\widehat{\operatorname{MTTF}}(n)$ with the approximations. The first-order entries are exact normal quantiles. The second-order entries are obtained from $5\times10^6$ simulations of the limiting Gaussian matrix. All entries are reported on the original MTTF scale. The curvature-refined approximation is no longer symmetric and captures the right-skewed deformation of the empirical law more closely for $n=500$ and $n=2000$.

\begin{table}[!htbp]
\centering
\small
\setlength{\tabcolsep}{3pt}
\caption{First-order Gaussian and curvature-refined approximations for the MTTF estimator on the original scale.}
\label{tab:second_order_mttf}
\begin{tabular}{c|ccc|ccc|ccc}
$n$ & \multicolumn{3}{c|}{Empirical} & \multicolumn{3}{c|}{First order} & \multicolumn{3}{c}{Curvature-refined}\\
& $2.5\%$ & $50\%$ & $97.5\%$ & $2.5\%$ & $50\%$ & $97.5\%$ & $2.5\%$ & $50\%$ & $97.5\%$\\
\hline
$500$ & $5.445$ & $6.903$ & $9.014$ & $5.178$ & $6.905$ & $8.632$ & $5.570$ & $6.911$ & $9.024$\\
$2000$ & $6.114$ & $6.904$ & $7.854$ & $6.041$ & $6.905$ & $7.768$ & $6.140$ & $6.906$ & $7.866$\\
$10000$ & $6.535$ & $6.904$ & $7.305$ & $6.519$ & $6.905$ & $7.291$ & $6.538$ & $6.905$ & $7.310$
\end{tabular}
\end{table}

For a direct comparison of the associated central $95\%$ confidence intervals, Table~\ref{tab:second_order_mttf_CI_differences} reports the signed discrepancies of the approximating endpoints from the empirical endpoints. Thus
\begin{equation*}
\Delta_L=q^{\rm app}_{0.025}-q^{\rm emp}_{0.025},
\qquad
\Delta_U=q^{\rm app}_{0.975}-q^{\rm emp}_{0.975},
\end{equation*}
and $\Delta_W$ is the corresponding difference of interval lengths. The first-order approximation places both endpoints too far to the left for $n=500$ and $n=2000$. The curvature-refined approximation corrects this displacement in the upper tail and gives a much closer upper endpoint. Its lower endpoint slightly overshoots for $n=500$, but the endpoint errors are already small for $n=2000$ and become negligible for $n=10000$.

\begin{table}[!htbp]
\centering
\caption{Differences between approximating and empirical central $95\%$ intervals for $\widehat{\operatorname{MTTF}}(n)$.}
\label{tab:second_order_mttf_CI_differences}
\begin{tabular}{c|c|rrr}
$n$ & Approximation & $\Delta_L$ & $\Delta_U$ & $\Delta_W$\\
\hline
$500$ & First order & $-0.267$ & $-0.382$ & $-0.115$\\
$500$ & Curvature-refined & $ 0.125$ & $ 0.010$ & $-0.115$\\
$2000$ & First order & $-0.073$ & $-0.086$ & $-0.013$\\
$2000$ & Curvature-refined & $ 0.026$ & $ 0.012$ & $-0.014$\\
$10000$ & First order & $-0.016$ & $-0.014$ & $ 0.002$\\
$10000$ & Curvature-refined & $ 0.003$ & $ 0.005$ & $ 0.002$
\end{tabular}
\end{table}

The improvement is most visible for $n=500$ and $n=2000$, where the nonlinearity of the inverse map still affects the upper tail. The curvature-refined approximation is not merely a translation of the normal curve: it reproduces the skewed deformation generated by the inverse functional. For $n=10000$, the first-order and curvature-refined approximations are close, as expected, and the empirical distribution is already very close to the limiting normal law.

\section{Discussion and perspectives}
\label{sec: discussion}
The abstract first-order transfer principle is established by the functional delta method. The contribution developed here is the finite Markov-chain calculus obtained after this transfer principle is combined with the Gaussian matrix representation of the transition-matrix MLE. A single matrix differential produces the limiting object in its natural form and, through row-wise vectorization, the covariance required for joint inference. The polynomial, inverse-type, stationary, reliability and long-run characteristics considered above are therefore treated by one construction rather than by separate coordinate systems. The finite-order developments further provide explicit curvature terms, as illustrated by the MTTF approximation.
The paper gives a unified matrix-level asymptotic calculus for plug-in MLEs in finite Markov models. The central object is the limiting Gaussian matrix of the estimated transition matrix. Once this object is identified, the asymptotic law of a Markov characteristic is obtained by differentiating the corresponding matrix functional and evaluating this derivative at the limiting Gaussian matrix. This is the common mechanism behind the formulas for powers, resolvents, stationary probabilities, curves and reliability indicators.

The practical advantage of the representation is that the limiting distribution is first obtained in the natural form of the characteristic itself: a matrix for a transition kernel, a vector for a curve, or a scalar for a reliability index. If a vector covariance matrix is needed, it is recovered afterwards by row-wise vectorization and Kronecker products. Thus the matrix and vector representations are not competing approaches; they are two forms of the same first-order expansion. The second-order development shows that the same calculus also identifies the curvature term of smooth functionals, providing a systematic approach to curvature-refined approximations and to the curvature component of order $n^{-1}$ mean corrections.

The explicit covariance formulas have immediate inferential consequences. Scalar characteristics lead to confidence intervals, finite-dimensional curves lead to simultaneous confidence bands or ellipsoidal regions, and linear restrictions lead to Wald-type tests. In this sense, the method provides a systematic approach from matrix representations of Markov-chain functionals to statistical inference for the corresponding plug-in estimators.

The finite-dimensional curve formulation is useful statistically. In applications one often estimates an entire time profile, such as an availability curve, a reliability curve or a cumulative occupation curve, rather than a single scalar characteristic. For any fixed time horizon, the present framework gives the joint Gaussian limit of the curve ordinates and their covariance matrix from the same operator formula. This viewpoint is also natural in settings where several sample paths are observed and the statistical object is a random curve.

A natural extension concerns semi-Markov chains. In that setting, the transition matrix is replaced by a semi-Markov kernel and the sojourn-time distributions enter the functionals; related reliability and hidden semi-Markov constructions may be found in \cite{votsi2014,votsi2015,DurandGaudoin2005,GamizPerez2023}. The same question then becomes whether the limiting Gaussian object can be organized so that differentiable functionals of the kernel admit matrix- or operator-level representations comparable to those derived here.
However, the semi-Markov kernel is infinite dimensional and the extension is not straightforward. In such settings, bootstrap procedures such as those in \cite{votsi2025} may also be useful for confidence intervals when closed covariance expressions become difficult to handle.

Another asymptotic regime consists of several independent trajectories of fixed or random finite length, with the number of trajectories tending to infinity; see \cite{trevezas2011} for related results in the semi-Markov context. This framework is common in survival analysis, biostatistics and longitudinal studies, where the observations are naturally viewed as sample paths or curves. In such a regime, the initial law is estimable jointly with the transition mechanism, and the joint asymptotic distribution of plug-in curve estimators becomes the relevant object.

\appendix
\section{Kronecker products and vectorization}
\label{app:kronecker}

We recall the elementary identities used in the paper; standard references include, for example, \cite[Chapter~4]{math_econ}. For matrices $\mathbf A\in\mathbb R^{m\times n}$ and $\mathbf B\in\mathbb R^{r\times q}$, the Kronecker product is the block matrix
\begin{equation*}
\mathbf A\otimes \mathbf B=
\begin{pmatrix}
\alpha_{11}\mathbf B &\rvline& \alpha_{12}\mathbf B &\rvline& \cdots &\rvline& \alpha_{1n}\mathbf B\\ \hline
\alpha_{21}\mathbf B &\rvline& \alpha_{22}\mathbf B &\rvline& \cdots &\rvline& \alpha_{2n}\mathbf B\\ \hline
\vdots &\rvline& \vdots &\rvline& \ddots &\rvline& \vdots\\ \hline
\alpha_{m1}\mathbf B &\rvline& \alpha_{m2}\mathbf B &\rvline& \cdots &\rvline& \alpha_{mn}\mathbf B
\end{pmatrix}.
\end{equation*}
It satisfies
\begin{equation*}
(\mathbf A\otimes\mathbf B)(\mathbf C\otimes\mathbf D)=(\mathbf{AC})\otimes(\mathbf{BD}),
\qquad
(\mathbf A\otimes\mathbf B)^\top=\mathbf A^\top\otimes\mathbf B^\top,
\end{equation*}
whenever the products are well defined.

\begin{lemma}
\label{lem: vec_propereties}
Let $\mathbf A\in\mathbb R^{p\times q}$, $\mathbf B\in\mathbb R^{q\times r}$ and $\mathbf C\in\mathbb R^{r\times m}$. Then
\begin{equation*}
\vecc(\mathbf A\mathbf B)=
(\mathbf I_r\otimes\mathbf A)\vecc(\mathbf B)
=(\mathbf B^\top\otimes\mathbf I_p)\vecc(\mathbf A),
\end{equation*}
and
\begin{equation*}
\vecc(\mathbf A\mathbf B\mathbf C)=
(\mathbf C^\top\otimes\mathbf A)\vecc(\mathbf B).
\end{equation*}
\end{lemma}

These identities are used only after the matrix differential has been found. Thus the limiting distribution is first obtained in the natural matrix, vector or scalar form of the Markov functional; the vector covariance matrix is then recovered by applying the appropriate vectorization identity.

\section*{Acknowledgements}
This work was partially supported by the ANR Project ``Hidden semi-Markov models: Inference, Control and Applications -- HSMM-INCA'' (ANR-21-CE40-0005).

\bibliographystyle{imsart-nameyear-issue}
\bibliography{sn-bibliography}

\end{document}